\documentclass[11pt,reqno]{amsart}
\pdfoutput=1
\usepackage{enumerate}
\usepackage{amssymb}
\usepackage{graphicx}

\catcode`\@=11

\long\def\@savemarbox#1#2{\global\setbox#1\vtop{\hsize\marginparwidth 
  \@parboxrestore\tiny\raggedright #2}}
\newcommand\lref[1]{\ref{#1}%
\@ifundefined{r@DisplaY #1}{}{ (#1)}}

\newcommand\fakelabel[2]{\@bsphack\if@filesw {\let\thepage\relax
   \newcommand\protect{\noexpand\noexpand\noexpand}%
\xdef\@gtempa{\write\@auxout{\string
      \newlabel{#1}{{#2}{\thepage}}}}}\@gtempa
   \if@nobreak \ifvmode\nobreak\fi\fi\fi\@esphack}

\def\SL@margintext#1{{\showlabelsetlabel{\tiny\{\SL@prlabelname{#1}\}}}}
\catcode`\@=12

\def\Empty{}
\newcommand\oplabel[1]{
  \def\OpArg{#1} \ifx \OpArg\Empty {} \else
        \label{#1}
  \fi}
\newtheorem{theoremSt}{Theorem}[section]

\newtheorem{exampleSt}[theoremSt]{Example}
\newtheorem{exerciseSt}[theoremSt]{Exercise}

\newcommand\MakeStEnv[1]{
  \newenvironment{#1}[1]{
  \begin{#1St} \oplabel{##1}%
  \global\def\CrntSt{\thetheoremSt}%
}{ 
  \end{#1St} }
  \newenvironment{#1+}[1]{
  \begin{#1St} \label{##1}%
  \label{DisplaY ##1}%
  \global\def\CrntSt{\thetheoremSt}%
  \def\Labl{##1}\ifx\Labl\Empty{} \else {\em (\Labl)\,}\fi%
}{ 
  \end{#1St} }
}
\MakeStEnv{theorem}
\MakeStEnv{corollary}
\MakeStEnv{proposition}
\MakeStEnv{lemma}
\MakeStEnv{definition}
\MakeStEnv{conjecture}
\MakeStEnv{problem}
\MakeStEnv{question}

\newlength{\saveu}

\newenvironment{pf*}[1]{%
 \begin{proof}[#1]%
}{ 
 \end{proof}
}

\newcommand{\finishproof}[1]{ 
  \def\FPArg{#1}
  \ifx\FPArg\Empty
        \newcommand\FPArg{\CrntSt}  \fi
  \smallbreak\noindent\makebox[\textwidth]{\hfill\fbox{\FPArg}}
  \medbreak\noindent
}

\newcommand\CC{{\mathcal C}}

\newcommand\FF{{\mathcal F}}

\newcommand\LL{{\mathcal L}}
\newcommand\MM{{\mathcal M}}

\newcommand\PP{{\mathcal P}}

\newcommand\UU{{\mathcal U}}

\newcommand\WW{{\mathcal W}}

\newcommand\PMF{{\PP\kern-2pt\MM\FF}}
\newcommand\ML{{\MM\LL}}
\newcommand\PML{{\PP\kern-2pt\MM\LL}}

\newcommand\ep{\epsilon}

\newcommand\hhat{\widehat}

\newcommand\union{\cup}

\newcommand\bbR{{\mathord{\text{I\kern-2pt R}}}}        
\newcommand\bbH{{\mathord{\text{I\kern-2pt H}}}}        

\newcommand\Hyp{{\mathbb H}}

\newcommand\bigrightarrow[1]{\hbox to #1{\rightarrowfill}}
\newcommand\bigleftarrow[1]{\hbox to #1{\leftarrowfill}}

\newcommand\boundary{\partial}
\newcommand\semidir{\mathrel{\hbox{\vrule depth-.03ex height1.1ex\kern-0.15em$\times$}}}

\newcommand\del{\nabla}

\newcommand{\diam}{\operatorname{diam}}

\numberwithin{equation}{section}

\catcode`\@=11

\def\subsection{\@startsection{subsection}{2}%
  \z@{.5\linespacing\@plus.7\linespacing}{.5em}%
  {\normalfont\bfseries\centering}}

\def\section{\@startsection{section}{1}%
  \z@{.7\linespacing\@plus\linespacing}{.5\linespacing}%
  {\normalfont\large\bfseries\centering}}

\def\subsubsection{\@startsection{subsubsection}{3}%
  \z@{.5\linespacing\@plus.7\linespacing}{-.5em}%
  {\normalfont\bfseries}}

\catcode`\@=12

\newcommand{\fsubd}{\mathrel{{\scriptstyle\searrow}\kern-1ex^d\kern0.5ex}}
\newcommand{\bsubd}{\mathrel{{\scriptstyle\swarrow}\kern-1.6ex^d\kern0.8ex}}
\newcommand{\fsubeq}{\mathrel{\raise-.7ex\hbox{$\overset{\searrow}{=}$}}}
\newcommand{\bsubeq}{\mathrel{\raise-.7ex\hbox{$\overset{\swarrow}{=}$}}}

\newcommand{\EL}{\mathcal{EL}}

\newcommand{\tsh}[1]{\left\{\kern-.9ex\left\{#1\right\}\kern-.9ex\right\}}

\newcommand\Teich{{\mathcal T}}

\newcommand{\ba}{{\boldsymbol \alpha}}

\begin{document}

\title[Local topology in deformation spaces II]{Local topology in deformation spaces of hyperbolic 3-manifolds II}

\author[Brock]{Jeffrey F. Brock}
\author[Bromberg]{Kenneth W. Bromberg}
\author[Canary]{Richard D. Canary}
\author[Lecuire]{Cyril Lecuire}
\author[Minsky]{Yair N. Minsky}
\date{\today}
\thanks{
Brock was partially supported by DMS-1608759, Bromberg was partially supported by DMS-1509171, Canary was partially supported by DMS-1306992 and DMS-1564362, and Minsky
was partially supported by DMS-1311844 and DMS-1610827 from the National Science Foundation. 
Lecuire was partially supported by  ANRproject GDSous/GSG no. 12-BS01-0003-01.}

\begin{abstract}
We prove that  the deformation space $AH(M)$ of marked hyperbolic 3-manifolds
homotopy equivalent to a fixed compact 3-manifold $M$  with
incompressible boundary is locally connected at  quasiconformally rigid points.
\end{abstract}

\maketitle

\setcounter{tocdepth}{1}

\newcommand\epzero{\ep_0}
\newcommand\epone{\ep_1}
\newcommand\epotal{\ep_{\rm u}}
\newcommand\kotal{k_{\rm u}}
\newcommand\Kmodel{K_0}
\newcommand\Kone{K_1}
\newcommand\Ktwo{K_2}
\newcommand\bdry{\partial} 
\newcommand\stab{\operatorname{stab}}
\newcommand\nslices[2]{#2|_{#1}}
\newcommand\ME{M\kern-4pt E}
\newcommand\bME{\overline{M\kern-4pt E}}
\renewcommand\del{\partial}
\newcommand\s{{\mathbf s}}
\newcommand\pp{{\mathbf p}}
\newcommand\qq{{\mathbf q}}
\newcommand\uu{{\mathbf u}}
\newcommand\vv{{\mathbf v}}
\newcommand\zero{{\mathbf 0}}

\newcommand{\cB}{{\mathcal B}}
\newcommand\mm{\operatorname{\mathbf m}}
\newcommand{\dehntw}{\theta}  

\section{Introduction}

The space $AH(M)$ of (marked) hyperbolic 3-manifolds homotopy
equivalent to a fixed compact orientable 3-manifold $M$ plays an important role
in the theory of 3-manifolds and is of interest in geometry and
dynamics in general, by way of analogy with other parameter spaces
such as those of conformal dynamical systems. The topology of this
space at its boundary points is quite intricate and remains poorly
understood after many years of study.

The main result of this paper is the following theorem proving local
connectivity at a natural class of boundary points of $AH(M)$,
whenever $M$ has incompressible boundary. This
generalizes previous work of Brock-Bromberg-Canary-Minsky
\cite{nobumping}, Ito \cite{ito-selfbump}  and
Ohshika \cite{ohshika-selfbump}. It
can be contrasted with work of Bromberg \cite{bromberg-PT} and Magid \cite{magid} showing that
$AH(M)$ is not everywhere locally connected when $M$ is an untwisted interval bundle.

\begin{theorem}{main}
If $M$ is a compact, hyperbolizable orientable 3-manifold with incompressible boundary and $\tau\in AH(M)$
is quasiconformally rigid, then $AH(M)$ is locally connected at $\tau$.
\end{theorem}

A hyperbolic 3-manifold $N$ is {\em quasiconformally rigid} if any hyperbolic \hbox{3-manifold} which is bilipschitz homeomorphic
to $N$ is actually isometric to $N$. Equivalently, $N$ is quasiconformally rigid if every component of its
conformal boundary is a thrice-punctured sphere. Note that the
conformal boundary may also be empty. See Section \ref{background} for detailed definitions.

Bromberg \cite{bromberg-PT} showed that the space of Kleinian
punctured torus groups fails to be locally connected at some, but not all, points whose conformal boundary
consists of one once-punctured torus and two thrice-punctured spheres. 
Therefore, one cannot expect any better result whose assumptions simply restrict the complexity
of the conformal boundary.

Our argument actually establishes the stronger fact
that components of the interior of $AH(M)$ cannot {\em self-bump} at a quasiconformally rigid point.
(We say that a component $C$ of the interior of $AH(M)$
self-bumps at $\rho$ if any sufficiently small neighborhood of $\rho$ in $AH(M)$ has disconnected
intersection with $C$.)

\begin{theorem}{nonselfbump}
Let $M$ be a compact, hyperbolizable 3-manifold with incompressible boundary. If $\tau\in AH(M)$
is quasiconformally rigid, then no component of ${\rm int}(AH(M))$ self-bumps at $\tau$.
\end{theorem}

\subsection*{History and prior results}

We recall briefly the history of the subject.
The interior of $AH(M)$ is well-understood due to
the work of Bers \cite{bers-spaces}, Kra \cite{kra}, Marden \cite{marden}, Maskit \cite{maskit}, Sullivan \cite{sullivanII} and
Thurston \cite{thurston-survey}. The components of the interior are enumerated by
marked homeomorphism types and each component is parameterized by analytic data (see Bers \cite{bers-survey} for a
survey of this theory in analytic language and Canary-McCullough
\cite[Chapter 7]{canary-mccullough} for a survey  in topological
language). The entire space $AH(M)$ is the closure of its interior 
(see Bromberg \cite{bromberg-density}, Brock-Bromberg \cite{brock-bromberg}, Brock-Canary-Minsky \cite{ELC2},
Namazi-Souto \cite{namazi-souto}, and Ohshika \cite{ohshika-density}.)
The Ending Lamination Theorem (see Minsky \cite{ELC1}, Brock-Canary-Minsky \cite{ELC2}, and Bowditch
\cite{bowditch})
provides a classification of hyperbolic 3-manifolds in $AH(M)$ in terms of ending invariants which record
the asymptotic geometry, but these ending invariants vary discontinuously (see Anderson-Canary \cite{ACpages}
and Brock \cite{brock-invariants}).    

We will focus on the case where $M$ has incompressible boundary. Anderson and Canary \cite{ACpages}
discovered that components of the interior of $AH(M)$ can {\em bump}, i.e. have intersecting closure,
and Anderson, Canary and McCullough \cite{ACM} characterized which components can bump in terms of
the topological change of homeomorphism type involved. McMullen \cite{mcmullenCE}, in the untwisted
interval bundle case, and Bromberg and Holt \cite{bromberg-holt}, more generally, discovered that  components
of the interior of $AH(M)$ can self-bump. Bromberg \cite{bromberg-PT} and Magid \cite{magid} showed that $AH(M)$ fails
to be locally connected when $M$ is an untwisted interval bundle and Bromberg conjectures that $AH(M)$
is never locally connected. See Canary \cite{canary-bumponomics} for a more complete discussion of the
topology of $AH(M)$.

In an earlier paper \cite{nobumping} we showed that there is no bumping or self-bumping at
a representation $\rho\in AH(M)$ if all parabolic elements of $\rho(\pi_1(M))$ 
lie in rank two abelian subgroups. If there is no bumping or self-bumping
at $\rho$ then we say that $\rho$ is {\em uniquely approachable} and we notice that $AH(M)$ is
locally connected at uniquely approachable representations.  
We further showed  that there is no bumping at quasiconformally rigid points (\cite[Thm. 1.2]{nobumping}), and that if $M$
is either an untwisted interval bundle or an acylindrical 3-manifold, then quasiconformally rigid
points are uniquely approachable (\cite[Cor. 1.4]{nobumping}). 
Theorems \ref{main} and  \ref{nonselfbump} extend these results to the general
incompressible-boundary case. 
We note that Ito \cite{ito-selfbump} and 
Ohshika \cite{ohshika-selfbump} obtained related results in the setting of quasifuchsian groups.

\subsection*{Sketch of proof}
The geometry of a hyperbolic 3-manifold is determined by its {\em end invariants}.
However, since these invariants do not vary continuously over $AH(M)$, one must
be extremely careful when using them to understand the topology of $AH(M)$.

Let $\tau$ be a quasiconformally rigid point on the boundary of a component $C$ of
${\rm int}(AH(M))$. Assume for simplicity that
$M$ is not an interval bundle, since there will be some additional complications in
this case, and that the boundary $\partial M$ of $M$ has no toroidal components.
We may assume, by replacing $M$ by a homotopy equivalent manifold,
that the quotient manifold $N_\tau=\mathbb H^3/\tau(\pi_1(M))$
is homeomorphic to ${\rm int}(M)$, so
the end invariants of $\tau$ consist of a curve system, called the {\em parabolic locus},
$p_\tau$ on $\partial M$, associated to the cusps of $N_\tau$, and a
lamination $\lambda_W$ filling each component $W$ of 
$\partial M \setminus p_\tau$ which is not a 3-holed sphere. 
Since we have previously shown that there is no bumping at $\tau$, our goal is
to show that $C$ does not self-bump at $\tau$. 

\subsubsection*{Neighborhood system}
Given $\rho\in C$, the ending data of $\rho$ is a conformal structure on $\boundary M$ and we
can consider its projection $\pi_W(\rho)$ to the {\em curve complex} $\ \CC(W)$ of
each component $W$ of $\boundary M \setminus p_\tau$. We define a
``neighborhood'' for $\tau$ in $C$ by choosing a neighborhood in each $\CC(W)$ of the
ending lamination component $\lambda_W$ and requiring the projections
$\pi_W(\rho)$ lie there, and also placing a bound
on the  length $\ell_\rho(\alpha)$ in $N_\rho$ of  each component
$\alpha$ of $p_\tau$.  Proposition \ref{nbhd system} shows that a sequence of
representations in $C$ converges to $\tau$ if and only it eventually lies in every such ``neighborhood.''
Equivalently, we say that sets of this form are the intersection with
$C$ of a neighborhood system for $\tau$ in $AH(M)$.
When $M$ is acylindrical, we established this in \cite[Lemma 8.2]{nobumping}.
 
The first new ingredient in our more general situation is a compactness theorem of Lecuire
\cite{lecuire-masurdomain} (Theorem \ref{doublyincomp}),
which is applied to show that a sequence of representations eventually contained in
any neighborhood of the above form must have a convergent
subsequence. If $\rho$ is a limit of such a sequence, 
one applies our earlier result on convergence of ending invariants (\cite[Thm. 1.3]{pull-out})
to exhibit a homotopy equivalence from $N_\tau$ to $N_{\rho}$ which
takes cusps to cusps and geometrically infinite ends to geometrically infinite ends
with the same ending laminations. A topological observation of Canary and Hersonsky 
\cite[Prop. 8.1]{canary-hersonsky} allows one to upgrade this homotopy equivalence
to a homemorphism and we conclude that $\tau=\rho$ by applying the Ending Lamination
Theorem \cite{ELC2}. The converse, that any sequence in $C$ converging to $\tau$ eventually lies
in any neighborhood of the form above, follows nearly immediately from \cite[Thm. 1.3]{pull-out}.

\subsubsection*{Skinning map and control of subsurface projections}
If $M=S\times [0,1]$ for a closed surface $S$, we use the shorthand $AH(S)$ for $AH(S\times [0,1])$.
The end invariant of a Kleinian surface group $\rho\in AH(S)$ consists of the
end invariant $\nu_+(\rho)$ on the top component $S\times \{1\}$ of $\partial (S\times [0,1])$
and the end invariant $\nu_-(\rho)$ of the bottom component $S\times \{0\}$.
Given a curve $\alpha$ on $S$, we define a 
quantity $\mm_\alpha(\nu_+(\rho),\nu_-(\rho))$ 
(see Section \ref{background}) which depends on the 
length of $\alpha$ in $\nu_+(\rho)$ and $\nu_-(\rho)$ and the distances $d_W(\nu_+(\rho),\nu_-(\rho))$ between projections
of $\nu_+(\rho)$ and $\nu_-(\rho)$ onto subsurfaces bordering $\alpha$. It follows from
work of  Minsky \cite{ELC1}  that
$\mm_\alpha(\nu_+(\rho),\nu_-(\rho))$ is ``large'' if and only if
$\ell_\rho(\alpha)$ is ``small'' (see Theorem \ref{KGCC results}).
This allows us to determine membership in the above neighborhood
system by studying the quantity $\mm_\alpha$.

In general, if  $\rho\in C\subset AH(M)$ and $S$ is a component of $\partial M$, we consider
the restriction $\rho^S=\rho|_{\pi_1(S)}$.  With
suitable orientation convention,
$\nu_+(\rho^S)$ is just the restriction of the
ending invariant of $\rho$ to $S$. The other ending invariant
$\nu_-(\rho^S)$ is determined more subtly ---
it is the image by Thurston's {\em skinning map} of the  full ending
invariant of $\rho$. Thus the control needed for determining if a
point is in the aforementioned neighborhoods of $\tau$ requires some
control over the skinning map.
In the acylindrical case, Thurston's Bounded Image Theorem provided this control.

In Proposition \ref{preliminary bounds} we show that if $\alpha$ is one of the curves in $p_\tau\cap S$
and $\rho$ is close enough to $\tau$, then the contribution of $\nu_-(\rho^S)$ to
$\mm_\alpha(\nu_+,\nu_-)$ is bounded from above.
More explicitly, if we fix a marking $\mu$ on $S$ and 
define $\mm_\alpha(\mu,\nu_-)$ similarly to $\mm_\alpha(\nu_+,\nu_-)$, 
then $\mm_\alpha(\mu,\nu_-(\rho^S))$
is bounded above in some neighborhood of $\tau$.
Proposition \ref{preliminary bounds} is a consequence of Theorem \ref{BBCL result}, 
a result on convergence of Kleinian surface groups proven by
Brock-Bromberg-Canary-Lecuire \cite{BBCL}.

\subsubsection*{Navigation in Teichm\"uller space}
The last main hurdle in establishing local connectivity is to show that,
within such neighborhoods, one can find paths in which the projection
data can be sufficiently well controlled. The tools for this are provided by our
earlier work \cite{nobumping}, where we  studied how subsurface projections behave 
as one deforms in Teichm\"uller space. With this in hand we rule out self-bumping by showing that for each
neighborhood $\UU$ of $\tau$ there is a smaller neighborhood $\UU'$
such that any two points in $\UU'$ can be connected by a path in
$\UU$, see Proposition \ref{joining nearby points}. 

\medskip\noindent 
{\bf Acknowledgements:} We thank the referee for helpful comments on an earlier version
of this paper. The authors  also gratefully acknowledge  support from U.S. National Science Foundation grants 
DMS 1107452, 1107263, 1107367 ``RNMS: GEometric structures And Representation varieties" (the GEAR Network).

\section{Background}
\label{background}

\subsection{Hyperbolic 3-manifolds and their deformation spaces}

We will assume throughout the paper that  $M$ is a compact, orientable, hyperbolizable 3-manifold 
with non-empty incompressible boundary which is not an interval bundle over the torus.
We recall that $M$ is {\em hyperbolizable} if its interior admits
a complete metric of constant negative curvature $-1$ and that $M$ has {\em incompressible boundary}
if whenever $S$ is a component of the boundary $\partial M$ of $M$, then $S$ is not a sphere and
the inclusion map of $S$ into $M$ induces an injection of $\pi_1(S)$ into $\pi_1(M)$. 
Let $\partial_0M$ denote the collection of non-toroidal components of $\partial M$.

Let $AH(M)$ denote the space of (conjugacy classes of) discrete, faithful representations of $\pi_1(M)$
into $\rm{PSL}(2,\mathbb C)$. We view $AH(M)$ as a quotient of a subset of the space
${\rm Hom}(\pi_1(M),\rm{PSL}(2,\mathbb C))$ of homomorphisms of $\pi_1(M)$ into $\rm{PSL}(2,\mathbb C)$.
We give ${\rm Hom}(\pi_1(M),\rm{PSL}(2,\mathbb C))$ the compact-open topology and $AH(M)$ inherits
the induced quotient topology.  In the case that \hbox{$M=S\times [0,1]$} and $S$ is a closed oriented surface,
we use the shorthand $AH(S)$ for $AH(S\times [0,1])$.

If  $\rho\in AH(M)$, then $N_\rho=\mathbb H^3/\rho(\pi_1(M))$ is a complete hyperbolic \hbox{3-manifold}. 
There exists a universal constant $\mu>0$, called the {\em Margulis constant}, so that if $\epsilon<\mu$, then
every component of 
$$(N_\rho)_{(0,\epsilon)}=\{x\in N_\rho\ | \ {\rm inj}(x)<\epsilon\}$$
(where ${\rm inj}(x)$ is the injectivity radius of $N_\rho$ at $x$), is either a {\em Margulis tube}, 
i.e. an open solid torus neighborhood of a closed geodesic in $N_\rho$ or a {\em cusp}, i.e. a quotient
of a horoball $H$ in $\mathbb H^3$ by a group of parabolic elements of $\rho(\pi_1(M))$  which preserve $H$.
We fix $\epsilon_0<\mu$ and let $N_\rho^0$ be obtained from $N_\rho$ by removing the cusps in
$(N_\rho)_{(0,\epsilon_0)}$. A {\em relative compact core} $M_\rho$ for $N_\rho^0$ is a compact submanifold of $N_\rho^0$
so that the inclusion of $M_\rho$ into $N_\rho$ is a homotopy equivalence and if $P_\rho=M_\rho\cap \partial N_\rho^0$
and $R$ is a component of $\partial N_\rho^0$, then the inclusion of $R\cap P_\rho$ into $R$ is a homotopy equivalence
(see Kulkarni-Shalen \cite{kulkarni-shalen} or
McCullough \cite{mcculloughRCC} for the existence of relative compact cores).
Bonahon \cite{bonahon} showed that $N_\rho$ is homeomorphic to the interior ${\rm int}(M_\rho)$ of $M_\rho$ 
and that each component of $N_\rho^0-{\rm int}(M_\rho)$ is bounded by a component $F$ of $\partial M_\rho-{\rm int}(P_\rho)$
and is homemorphic to $F\times [0,\infty)$.

There is a homotopy equivalence $h_\rho:M\to M_\rho$, well-defined up to homotopy, in the homotopy class
determined by $\rho$.
It will often be natural to restrict to the subset $AH_0(M)$ where $h_\rho$ is homotopic to an orientation-preserving
homeomorphism.
If $h:M\to M'$ is a homotopy equivalence, then
$h$ induces an identification of $AH(M)$ with $AH(M')$. In particular, if $\rho \in AH(M)$, then we may identify
$AH_0(M)$ with $AH(M_\rho)$, so that $\rho\in AH_0(M_\rho)$.  So, in our study of the local topology of $AH(M)$,
it does not reduce generality to assume that we are in a neighborhood of a 
representation in $AH_0(M)$. Moreover, a quasiconformally rigid representation in $AH_0(M)$ has a neighborhood in $AH(M)$
which is entirely contained in $AH_0(M)$, see Theorem \ref{no bump} below.

\subsection{Ending invariants}

If $W$ is a compact orientable hyperbolizable surface with boundary, other than the disk or annulus, then
the vertex set of 
its {\em curve complex}  $\ \mathcal C(W)$ is the set of (isotopy classes of) simple, closed, non-peripheral
curves on $W$.  If $W$ is not a one-holed torus or a four-holed sphere, 
a collection $\{\alpha_0,\alpha_1,\ldots,\alpha_n\}$ of vertices of $\mathcal C(W)$ span a
$n$-simplex in $\mathcal C(W)$ if and only if the (isotopy classes of) curves have mutually disjoint representatives. 
Masur and Minsky \cite{masur-minsky} proved that $\mathcal C(W)$ is  Gromov hyperbolic and Klarreich \cite{klarreich}
(see also Hamenst\"adt \cite{hamenstadt})
identified its Gromov boundary $\partial_\infty \mathcal C(W)$ with the space of filling  minimal geodesic laminations
on $W$ which are the support of a measured lamination. Recall that, given a background finite area metric on $W$,
a {\em geodesic lamination} is a closed subset which is a disjoint union of geodesics. A geodesic lamination is
{\em filling} if it intersects every closed geodesic on $W$ essentially. A {\em measured lamination} is a geodesic
lamination together with a transverse measure, i.e. an assignment of a measure to each arc transverse
to the lamination which is invariant under isotopies preserving the lamination.

Let $p_\rho$ be the core curves
of the annular components of $P_\rho=M_\rho\cap N_\rho^0$. Then $p_\rho$ is called the {\em parabolic locus} of
$\rho$ and is a well-defined isotopy class of curves in $\partial_0M_\rho$.

Let $\Omega(\rho)$ denote the
largest open subset of $\partial_\infty\mathbb H^3=\widehat{\mathbb C}$ which $\rho(\pi_1(M))$ acts properly
discontinuously on. Then
$$\partial_cN_\rho=\Omega(\rho)/\rho(\pi_1(S))$$
is a Riemann surface, called the {\em conformal boundary}, and \hbox{$\widehat N_\rho=N_\rho\cup \partial_cN_\rho$}
is a 3-manifold with boundary.
We say that $\rho$ is {\em quasiconformally rigid} if every component of the conformal boundary is a
thrice-punctured sphere.

If $W$ is a component of $\partial_0M_\rho-p_\rho$,
then either $W$ is parallel to a component of $\partial_cN_\rho$, in $\widehat{N_\rho}-M_\rho$, in which case we say that $W$ bounds
a geometrically finite end, and one obtains a well-defined
finite area conformal structure on $W$ from the parallel component of $\partial_cN_\rho$,
or $W$ bounds a geometrically infinite end. If $W$ bounds a geometrically
infinite end, then there exists a sequence $\{\alpha_n\}$ of simple closed curves on $W$ whose geodesic
representatives $\{\alpha_n^*\}$ in $N_\rho$ intersect the component of $N_\rho-M_\rho$ bounded by $h_\rho(W)$ and exit
every compact subset of $N_\rho$.  In this case, the sequence $\{\alpha_n\}$ converges to a point
in $\partial_\infty\mathcal C(W)$, which is called the {\em ending lamination} of $W$ in $N_\rho^0$. 

The {\em ending invariant} $\nu(\rho)$ of $\rho$ consist of the parabolic locus $p_\rho$ on $M_\rho$, the conformal structures on
components of $\partial_0M_\rho-p_\rho$ which bound a geometrically finite end, and the ending laminations
on the components which bound geometrically infinite ends. The Ending Lamination Theorem \cite{ELC1,ELC2,bowditch}
implies that if $\rho,\sigma\in AH(M)$ and there exists a homeomorphism $g:M_\rho\to M_\sigma$, in the homotopy
class of $\sigma\circ\rho^{-1}$ which takes the end invariants of $\rho$ to the end invariants of $\sigma$, then
there exists an isometry $G:N_\rho\to N_\sigma$ in the homotopy class of $\sigma\circ\rho^{-1}$ which takes
ends in $N_\rho^0$ to the ends of $N_\sigma^0$ with the corresponding ending data. 
See the discussion in Minsky \cite[Sec. 2]{ELC1} for more details on ending invariants.

\subsection{Topology of deformation spaces}
The topology of the interior  ${\rm int}(AH(M))$ of $AH(M)$ is well understood. 
Sullivan \cite{sullivanII} proved that $\rho\in AH(M)$ lies in the interior of $AH(M)$ if and only if the ending
invariants consist entirely of conformal structures on non-toroidal boundary components of $\partial M_\rho$, i.e.
$N_\rho^0$ has no geometrically infinite ends or annular boundary components corresponding to rank one cusps.
If $\rho\in {\rm int}(AH_0(M))$, then we may identify $M_\rho$ with $M$ and
view the end invariant $\nu(\rho)$ as an element of $\mathcal T(\partial_0M)$. Work of 
Bers \cite{bers-spaces}, Kra \cite{kra}, Marden \cite{marden}, and Maskit \cite{maskit} implies that this identification
induces a homeomorphism between ${\rm int}(AH_0(M))$
and $\Teich(\partial_0M)$. In general, if $\rho$ lies in a component $C$ of ${\rm int}(AH(M))$,
one obtains an identification of $C$ with $\mathcal T(\partial M_\rho)$, but $M_\rho$ need not be homeomorphic to $M$
(see \cite{canary-mccullough} for a more detailed discussion of this parameterization and its history).
Brock, Canary and Minsky \cite{ELC2}
proved that when $M$ has incompressible boundary, then $AH(M)$ is the closure of ${\rm int}(AH(M))$, and
Namazi-Souto \cite{namazi-souto} and Ohshika \cite{ohshika-density} established the same fact when $M$ is
any compact hyperbolizable manifold. However, Anderson and Canary \cite{ACpages} showed that the closure
of ${\rm int}(AH_0(M))$ need not be entirely contained in $AH_0(M)$.

In our earlier paper, we investigated the ``bumping locus'' of $AH(M)$ and showed that $AH(M)$ cannot bump
at a quasiconformally rigid point $\rho$, i.e. a quasiconformally rigid point is in the closure of exactly
one component of ${\rm int}(AH(M))$. (In fact, the result there holds without the restriction that $M$ have incompressible
boundary.)

\begin{theorem}{no bump}{\rm (\cite[Thm. 1.2 and Prop. 3.2]{nobumping})}
If $M$ is a compact, orientable hyperbolizable 3-manifold with incompressible boundary
and  $\rho_0\in \partial AH(M)$ is quasiconformally rigid,
then $AH(M)$ cannot bump at $\rho$. Moreover, 
there exists a neighborhood $U$ of $\rho_0$ in $AH(M)$, so that if $\rho\in U$, then there
exists an orientation-preserving homeomorphism $h:N_{\rho}\to N_{\rho_0}$ in the homotopy
class of $\rho_0\circ \rho^{-1}$.
\end{theorem}

We will make crucial use of a key tool in the  proof of this result. Recall that a sequence $\{\Gamma_n\}$
of Kleinian groups is said to converge {\em geometrically} to a Kleinian group $\Gamma$
if $\{\Gamma_n\}$ converges to $\Gamma$ in the Chabauty topology on closed subsets of $\rm{PSL}(2,\mathbb C)$.
If  a sequence $\{\rho_n\}$ converges to $\rho$ in $AH(M)$, then we may choose representatives in 
${\rm Hom}(\pi_1(M),\rm{PSL}(2,\mathbb C))$, still called $\{\rho_n\}$ and $\rho$, so that 
$\{\rho_n\}$ converges to $\rho$ in ${\rm Hom}(\pi_1(M),\rm{PSL}(2,\mathbb C))$ and $\{\rho_n(\pi_1(M))\}$
converges geometrically to a torsion-free Kleinian group $\hat\Gamma$ which contains $\rho(\pi_1(M))$,
see J\o rgenson-Marden \cite[Prop. 3.8]{jorgenson-marden}. In this
situation, there is a natural covering map from $N_\rho$ to $\hat N=\mathbb H^3/\hat\Gamma$.
If $\rho$ is quasiconformally rigid, then this covering map restricts to an embedding on some
relative compact core for $N_\rho^0$:

\begin{proposition}{convexembed}{\rm (Anderson-Canary-Culler-Shalen \cite[Prop. 3.2, Remark 3.3]{ACCS})}
If $\rho$ is a quasiconformally rigid point in $\partial AH(M)$
and $\{\rho_n\}$ converges to $\rho$
and $\{\rho_n(\pi_1(M))\}$ converges geometrically to $\hat\Gamma$, then there is
a compact core $M_\rho$ for $N_\rho$ which embeds in $\widehat N=\Hyp^3/\hat\Gamma$
under the obvious covering map.
\end{proposition}

\noindent
{\bf Remark:} The results of \cite{ACCS} actually show that the convex core $C(N_\rho)$ of $N_\rho$ embeds in $\widehat N$.
One may then simply take $M_\rho$ to be a  compact core for $C(N_\rho)$, which is hence
a compact core for $N_\rho$, as in the proof of \cite[Prop. 3.2]{nobumping}, to obtain our Proposition \ref{convexembed}.

\subsection{Subsurface projections and Kleinian surface groups}

It is a central ingredient in the proof of the Ending Lamination Theorem 
that the subsurface projections of the ending invariants of a Kleinian surface group $\rho\in AH(S)$
coarsely determine the geometry of $N_\rho$. We recall several explicit forms of this crucial principle.

If $\rho\in AH(S)$, then its end invariant decomposes as a pair $(\nu_+(\rho),\nu_-(\rho))$
where $\nu_+(\rho)$ is the ending invariant of the upward pointing end of $N_\rho$.
We recall that $AH(S)=AH_0(S\times [0,1])$ and that there exists an orientation-preserving
homeomorphism $h_\rho:S\times [0,1]\to M_\rho$ and we call $h_\rho(S\times \{1\})$ the upward-pointing
component of $\partial M_\rho$ and call its associated end invariant upward-pointing. Curves in
the parabolic locus of the upward-pointing end invariant are called upward-pointing parabolic curves.

If $W$ is an essential subsurface of $S$, one may (coarsely) define a subsurface projection 
$\pi_W:\mathcal E(S)\to \mathcal C(W)$
where $\mathcal E(S)$ is the collection of possible ending invariants on $S$. If $W$ is not an annulus and
$\nu\in \mathcal T(S)$ is a conformal structure on $S$, then $\pi_W(\nu)$ is obtained by considering a
shortest curve, in the induced hyperbolic structure on $S$, which essentially intersects $W$ and surgering
it with $\partial W$ to obtain an element of $\mathcal C(W)$. 
There are several choices involved in this construction, but there is a uniform upper bound on the distance
between any two such curves, and we simply choose one of the curves obtained in this manner.
One must take more care in the case
that $\nu$ is a general ending invariant or $W$ is an incompressible annulus in $S$, see
Minsky \cite[Section 4]{ELC1} for a complete discussion.

If $\alpha$ is a simple closed curve on $S$,  and $\nu$ and $\mu $ are ending invariants in $\mathcal E(S)$,
we define
$$\mm_\alpha(\nu,\mu)=\max\left\{
\frac{1}{\ell_\alpha(\nu)},\frac{1}{\ell_\alpha(\mu)},
\sup_{\alpha\subset\boundary W} d_W(\nu,\mu)\right\}$$
where  
$$d_W(\nu,\mu)={\rm diam}_{\mathcal C(W)}(\pi_W(\nu)\union\pi_W(\mu))$$
and the supremum in the final term is taken over all  essential subsurfaces $W$ with $\alpha$
in their boundary. Here if $\nu$ is an ending invariant on $S$, we define
$1/\ell_\alpha(\nu)$ as follows: 
If $\alpha$ lies in a component
of the complement of the parabolic locus $p_\nu$ of $\nu$ which admits a conformal structure,
$\ell_\alpha(\nu)$ is the length of
the geodesic representative of $\alpha$ in the associated hyperbolic metric. If
$\alpha$ is a component of $p_\nu$
 we define  $\frac{1}{\ell_\alpha(\nu)}=+\infty$. In the remaining case, $\alpha$ either lies in a
 component of the complement of $p_\nu$ associated to a geometrically infinite end or intersects $p_\nu$ non-trivially, and we set
$\frac{1}{\ell_\alpha(\nu)}=0$.

The Length Bound Theorem from \cite{ELC2} shows that the end invariants coarsely determine the
set of ``short'' curves in $N_\rho$ and their length. The following simplified version of the Length Bound
Theorem is a mild generalization of Theorem 2.2 in our previous paper \cite{nobumping}. 

\begin{theorem}{KGCC results} {\rm (\cite{ELC2})}
Suppose that $S$ is a compact, oriented, hyperbolic surface and $\rho\in AH(S)$.
\begin{enumerate}
\item
Given $\delta>0$, there exists $K=K(\delta,S)$, so that if 
\hbox{$\mm_\alpha(\nu_+(\rho),\nu_-(\rho))>K$}, then $\ell_\alpha(\rho)<\delta$.
\item
Given $K>0$, there exists $\epsilon=\epsilon(K,S)>0$ so that if
$\ell_\alpha(\rho)<\epsilon$, then $\mm_\alpha(\nu_+(\rho),\nu_-(\rho))>K$.
\end{enumerate}
\end{theorem}

The results of \cite{BBCL} give a relatively complete picture of the relationship between the asymptotic
behavior of the ending invariants of a convergent sequence of Kleinian surface groups and the
geometry of the algebraic limit. We state  results from \cite{BBCL} in the simpler case where one
knows that a compact core of the algebraic limit embeds in the geometric limit.
Proposition \ref{convexembed} will assure us that we are always in this simpler case when considering
a sequence which converges to a quasiconformally rigid point.

Recall that a {\em complete marking} of $S$ is a maximal collection $\{\alpha_1,\ldots,\alpha_i\}$ 
of disjoint simple closed curves on $S$, together with a collection $\{\beta_1,\ldots,\beta_i\}$ of simple closed
curves such that $\beta_j$ is disjoint from $\alpha_i$ if $j\ne i$ and intersects $\alpha_j$ twice if $\alpha_j$ separates
$S-\cup_{i\ne j}\alpha_i$ and once if it does not separate.
If $\mu$ is a complete marking of $S$ and $W$ is an essential subsurface of $S$, then we may obtain
$\pi_W(\mu)\in\mathcal C(W)$ by surgering a curve in $\mu$ which intersects $W$ essentially with $\partial W$.
Again, the annulus case is slightly more complicated, see \cite[Section 5.1]{ELC1} for more details.
If $\nu\in\mathcal E(S)$ and $\mu$ is a complete marking on $S$, then we define
$$\mm_\alpha(\nu,\mu)=\max\left\{
\frac{1}{\ell_\alpha(\nu)},
\sup_{\alpha\subset\boundary W} d_W(\nu,\mu)
\right\}$$
where $d_W(\nu,\mu)=d_{\mathcal C(W)}(\pi_W(\nu),\pi_W(\mu))$
and the supremum in the final term is taken over all  essential subsurfaces $W$ with $\alpha$
in their boundary.

It is an elementary exercise to verify that $\mm_\alpha$ satisfies the
triangle inequality.

The following result combines  portions of the two main results of \cite{BBCL} in our simpler setting.
Notice that Lemmas 4.3 and 4.4 in \cite{BBCL} assure that we do not have to pass to a further
subsequence as in the statements of Theorems 1.1 and 1.2 in \cite{BBCL}. Moreover,  $\ell_\rho(\alpha)>0$ if and only if
$\{\mm_\alpha(\nu_+(\rho_n),\mu)\}$ and $\{\mm_\alpha(\nu_-(\rho_n),\mu)\}$
are both eventually bounded, and part (5) of
\cite[Thm. 1.2]{BBCL} guarantees that, in the  setting of Theorem \ref{BBCL result}, 
either $\{\mm_\alpha(\nu_+(\rho_n),\mu)\}$ or $\{\mm_\alpha(\nu_-(\rho_n),\mu)\}$
is eventually bounded.
Here we say a sequence $\{x_n\}$ in $[0,\infty]$ is {\em eventually bounded} if $\limsup \{x_n\}<\infty$.

\begin{theorem}{BBCL result}{\rm (\cite[Thm. 1.1, Thm. 1.2]{BBCL})}
Suppose that $\{\rho_n\}$ is a sequence in $AH(S)$ converging to
$\rho\in AH(S)$,  $\{\rho_n(\pi_1(S))\}$ converges geometrically to $\hat\Gamma$
and there is a compact core for $N_\rho$ that embeds in $\widehat N={\bf H}^3/\hat\Gamma$.
If $\alpha$ is an upward-pointing parabolic for $\rho$ and $\mu$ is a complete marking on $S$, then 
$\{\mm_\alpha(\nu_-(\rho_n),\mu)\}$ is eventually bounded.
\end{theorem}

\medskip

The following result from \cite{pull-out} allows us to control the development of geometrically
infinite ends and the resulting ending laminations.

\begin{theorem}{endlams in limit}{\rm (\cite[Thm. 1.3]{pull-out})}
Suppose that $\{\rho_n\}$ is a sequence in $AH(S)$ converging to $\rho\in AH(S)$.
If $W \subseteq S$ is an essential subsurface of $S$, other than an annulus or a pair of pants,
and $\lambda\in\EL(W)$ is a lamination supported on $W$, the following statements are
equivalent:
\begin{enumerate}
\item
$\lambda$ is a component of $\nu_+(\rho)$.
\item
$\{\pi_W(\nu_+(\rho_n))\}$ converges to $\lambda$.
\end{enumerate}
\end{theorem}

\section{Neighborhood systems for quasiconformally rigid points}

In this section, we produce a neighborhood system for a quasiconformally
rigid representation $\tau\in AH_0(M)$. 
We recall that  Theorem \ref{no bump} implies that  $\tau$ lies 
in the boundary of ${\rm int}(AH_0(M))$ and lies in the boundary of no other component of ${\rm int}(AH(M))$.

Let $\mathcal W_\tau$ denote the collection of components of $\partial_0M_\tau-P_\tau$
which are not thrice-punctured spheres.
(Here, and throughout, we identify $M_\tau$ with $M$ and hence $P_\tau$ with a collection of
incompressible annuli and tori in $\partial_0M$.)
Let $p_\tau$ denote the multicurve of cores of annulus components of $P_\tau$.
If $W\in \mathcal W_\tau$, since $\tau$ is quasiconformally rigid the
end invariant associated to $W$ is an ending lamination
$\lambda_W\in\partial_\infty\mathcal C(W)$. If $W\in \mathcal W_\tau$ is contained in a component $S$ of $\partial_0M$
and $\rho\in AH_0(M)$,
then $\pi_W(\nu(\rho))$ will denote $\pi_W(\nu(\rho)|_S)$ where $\nu(\rho)|_S$ is the restriction of
$\nu(\rho)$ to $S$.

Given $\delta>0$ and a collection $\mathbb U=\{U_W \}_{W\in \mathcal W_\tau}$  so that each $U_W$
is a neighborhood of
$\lambda_W\in\partial_\infty\mathcal C(W)$ in $\ \mathcal C(W)$, we define
$\mathcal U(\delta,\mathbb U,\tau)$  to be the set of all \hbox{$\rho\in {\rm int}(AH_0(M))$} such that
\begin{enumerate}
\item
$\ell_\rho(\alpha)<\delta$ if $\alpha\in p_\tau$, and
\item
$\pi_{W}(\nu(\rho))\in U_W\in\mathbb U$ if $W\in \mathcal W_\tau$.
\end{enumerate}
If $\tau$ is  a maximal cusp, meaning that $p_\tau$ is a pants
decomposition of $\partial_0M$, then $\WW_\tau=\emptyset$ and hence $\mathbb U = \emptyset$, and we also write
$\UU(\delta,\tau) = \UU(\delta,\emptyset,\tau)$.

If $\rho\in AH(M)$, then $\bar\rho\in AH(M)$ is obtained from $\rho$ by complex conjugation, i.e.
conjugation by $z\to \bar z$. Notice that $N_{\bar\rho}$ is simply $N_\rho$ with the opposite orientation.
One may check that $\rho=\bar\rho$ if and only if $M$ is an interval bundle and $\rho$ is virtually Fuchsian
i.e. $\rho(\pi_1(M))$ has a finite index subgroup conjugate into ${\rm PSL}(2,\mathbb R)$ (but we will not
use this fact).

If $M$ is an interval bundle, then there exists an orientation-reversing involution $\iota_M:M\to M$ which
preserves each fiber and is homotopic, but not isotopic, to the identity. In this case,  $AH(M)=AH_0(M)$, and  if $\nu$ is the end invariant for
$\rho\in AH(M)$, then $i_M(\nu)$ is the end invariant for $\bar\rho$.

The following proposition is the main result of this section, stating
that sets of the form $\UU(\delta,\mathbb U,\tau)$ form a neighborhood
system for $\tau$ (in the case that $M$ is an interval bundle and $\tau$ is a maximal cusp,
we actually obtain a neighborhood system for the pair $\{\tau,\bar\tau\}$). 

\begin{proposition}{nbhd system}
Suppose that  $M$ has incompressible boundary and $\tau$ is a quasiconformally rigid representation in $AH_0(M)$.
If $M$ is not an interval bundle or $\mathcal W_\tau$ is non-empty,
then
the collection of sets of the form  $\mathcal U(\delta,\mathbb U, \tau)$ is the intersection 
of a local neighborhood system for $\tau$ in $AH(M)$ with ${\rm int}(AH_0(M))$.

If $M$ is an interval bundle and $\mathcal W_\tau$ is empty, then 
the collection of sets of the form  $\mathcal U(\delta,\tau)$ is the intersection with ${\rm int}(AH_0(M))$
of a local neighborhood system for $\{\tau,\bar\tau\}$ in $AH(M)$.
\end{proposition}

It suffices to prove that a sequence $\{\rho_n\}$ in ${\rm int}(AH_0(M))$
converges to $\tau$ (or accumulates on $\{\tau,\bar\tau\}$ when $M$ is an interval bundle and $\mathcal W_\tau$ is empty)
if and only if it is eventually contained in any set of the form $\mathcal U(\delta,\mathbb U,\tau)$. 

One direction follows easily from Theorem \ref{endlams in limit}.  Notice that if $M$ is an interval bundle and
$\mathcal W_\tau$ is empty, then $\mathcal U(\delta,\tau)=\mathcal U(\delta,\bar\tau)$ for all $\delta>0$.

\begin{lemma}{conv in nbhd}
Suppose that  $M$ has incompressible boundary,
$\tau$ is a quasiconformally rigid representation in $AH_0(M)$ and
$\{\rho_n\} \subset {\rm int}(AH_0(M))$ converges to $\tau$.
If $\delta>0$ and $\mathbb U=\{U_W \}_{W\in \mathcal W_\tau}$ is a collection of neighborhoods
of the ending laminations of $N_\tau^0$,
then $\rho_n$ is contained in $\mathcal U(\delta,\mathbb U,\tau)$ for all sufficiently large $n$.
\end{lemma}

\begin{proof}
If $\alpha\in p_\tau$, then $\ell_\alpha(\tau)=0$,
so, since $\lim\ell_{\alpha}(\rho_n)=\ell_\alpha(\tau)$, \hbox{$\ell_{\alpha}(\rho_n)<\delta$} for all sufficiently
large $n$. Fix $W\in \mathcal W_\tau$ and let $S$ be the component of
$\partial M$ containing $W$. Let $\rho_n^S=\rho_n|_{\pi_1(S)}$ and $\tau^S=\tau|_{\pi_1(S)}$.
Since $\{\rho_n^S\}$ converges to $\tau^S$ in $AH(S)$,
and the geometrically infinite end associated to $W$ is upward-pointing,
Theorem \ref{endlams in limit} implies that $\{\pi_W(\nu(\rho_n))\}$ converges to $\lambda_W$. 
(Recall that with this convention $\nu_+(\rho_n^S)$ is the restriction of $\nu(\rho_n)$ to $S$.)
In particular, 
$\pi_W(\nu(\rho_n))$ is contained in $U_W\in\mathbb U$ for all sufficiently large $n$. 
Therefore, $\rho_n$ is contained in $\mathcal U(\delta,\mathbb U,\tau)$ for all sufficiently large $n$.
\end{proof}

In order to establish the other direction of our claim, we will need the following (slight generalization of a) criterion for
convergence due to Bonahon and Otal \cite[Lem. 14]{bonahon-otal}. This criterion is a common generalization 
of Thurston's Double Limit Theorem \cite{thurstonII} and Relative Boundedness Theorem \cite[Thm 3.1]{thurstonIII}. 
It was generalized to manifolds with compressible boundary by Kleineidam-Souto \cite{kleineidam-souto} and 
Lecuire \cite[Thm. 6.6]{lecuire-masurdomain}. We will explain how to modify the proof of \cite[Thm. 6.6]{lecuire-masurdomain} 
to obtain the precise statement we need.

A measured lamination $\mu$ on $\partial_0 M$ is  {\em doubly incompressible}
if there exists $c>0$ such that if $E$ is an essential annulus or M\"obius band in $M$, then 
\hbox{$i(\partial A,\mu)\ge c$}.  (Recall that a properly embedded annulus or M\"obius band is {\em essential} if it
is $\pi_1$-injective and can not be properly homotoped into the boundary of $M$.)

\begin{theorem}{doublyincomp}
Let $M$ be a compact, orientable, hyperbolizable 3-manifold with incompressible boundary. Suppose that $\{\rho_n\}$ is a
sequence in $AH(M)$ and there exists a sequence $\{\mu_n\}$ of measured laminations on $\partial_0 M$ such that
\begin{enumerate}
\item $\{\ell_{\mu_n}(\rho_n)\}$ is bounded,
\item $\{\mu_n\}$ converges to a doubly incompressible measured lamination $\mu$ on $\partial_0 M$, and
\item the supports of $\{\mu_n\}$ Hausdorff converge to a geodesic lamination $L$ and every component $L_0$  of $L$
contains a unique minimal sublamination $L_0^m$. 
\end{enumerate}
Then $\{\rho_n\}$ has a convergent subsequence.
\end{theorem}

\begin{proof}
We first show, given our assumptions on $\mu$ and $L$, that if \hbox{$\pi_1(M)\times \mathcal T\to\mathcal T$} is a small, minimal
action of $\pi_1(M)$ on an $\mathbb R$-tree $\mathcal T$, then some component $L_0$ of $L$ is realized in $\mathcal T$
(see Lecuire \cite{lecuire-masurdomain} and Otal \cite{otal} for the definitions of the terminology from the theory of $\mathbb R$-trees
used here.)
Lecuire  \cite[Prop. 6.1]{lecuire-masurdomain} showed that, since $\mu$ is doubly incompressible, some component  $\mu_0$ of $\mu$ 
is realized in $\mathcal T$. Let $S$ be the component of $\partial_0M$ containing $\mu_0$ and let 
$\mathcal{ T}_S$ be a minimal subtree of $\mathcal {T}$ invariant under the action of $\pi_1(S)\subset \pi_1(M)$.
By Skora's Theorem \cite{skora}, the action of $\pi_1(S)$ on $\mathcal{T}_S$ is dual to a measured geodesic lamination $\beta$ on $S$.
Since $\mu_0$ is realized in $\mathcal{T}$, and hence in $\mathcal{T}_S$, $\mu_0$ intersects $\beta$ transversely. 
Let $L_0$ be the component of $L$ containing the support of $\mu_0$. 
Assumption (3)  guarantees that $L_0$ is obtained from the support of $\mu_0$ by adding finitely many isolated non-compact leaves.
Results of Otal \cite[Corol. 3.1.3 and Thm. 3.1.4]{otal} then imply that $L_0$ is realized in $\mathcal{T}_S$ and hence in $\mathcal{T}$. 

One may now complete the proof of Theorem \ref{doublyincomp} exactly as in the proof of \cite[Prop. 6.6]{lecuire-masurdomain}.
If $\{\rho_n\}$ does not have a convergent subsequence, then there 
exists a subsequence converging  to a small, minimal action of $\pi_1(M)$ on a $\mathbb R$-tree $\mathcal T$
(see Morgan-Shalen \cite[Thm. II.4.7]{morgan-shalenI}).
Assumption (3) implies that there exist sublaminations $\{\hat\mu_n\}$ of $\{\mu_n\}$ so that $\{\hat\mu_n\}$ 
converges to $\mu_0$ and the supports of $\{\hat\mu_n\}$ Hausdorff converge to $L_0$.
Lecuire's version \cite[Thm. 6.5]{lecuire-masurdomain} of
Otal's Continuity Theorem (see \cite[Thm. 4.0.1]{otal} and \cite[Thm. 3.1]{otal-duke}), then implies, since $L_0$ is realized in $\mathcal T$,
that if $\{\gamma_n\}$ is a sequence of
multi-curves which Hausdorff converge to $L_0$ and $\sigma_0:\pi_1(S)\to {\rm PSL}(2,\mathbb R)$ is a fixed Fuchsian representation,
then 
$$\frac{\ell_{\gamma_n}(\rho_n)}{\ell_{\gamma_n}(\sigma_0)}\to\infty.$$
Since we may approximate each $\hat\mu_n$ arbitrarily closely by laminations supported on multicurves,
\hbox{$\ell_{\hat\mu_n}(\rho_n)\to\infty$}, 
which contradicts Assumption (1). This contradiction completes the proof.
\end{proof}

It is well-known that the union of the parabolic locus and the ending laminations of a quasiconformally rigid representation is
doubly incompressible, see for example the discussion in Anderson-Lecuire \cite[Section 2.9]{anderson-lecuire}. We
include a proof since we could not find a complete argument in the literature.

\begin{lemma}{end lam doubly}
Suppose that  $M$ has incompressible boundary and that $\tau$ is a quasiconformally rigid representation in $AH_0(M)$.
If $\mu$ is a measured lamination on $\partial_0 M$ whose support is 
$$\lambda_\tau=p_\tau \cup \bigcup_{W\in \mathcal W_\tau} \lambda_W $$
then $\mu$ is doubly incompressible.
\end{lemma}

\begin{proof} 
First recall that if $A$ is an essential annulus or M\"obius band in $M$, then
$\partial A\cap \partial_0M$ cannot be homotoped into $p_\tau$ (since the fundamental groups of
the cusps represent distinct conjugacy classes of maximal abelian subgroups).
Since every simple closed curve in a pair of pants is peripheral and each ending lamination
fills a component of $\partial_0M-p_\tau$, every essential annulus or
M\"obius band $A$ in $M$ must intersect the support
of $\mu$. In particular $i(\partial A,\mu)>0$.
So, if $\mu$ is not doubly incompressible, there exists a sequence $\{A_i\}$ of distinct essential annuli or M\"obius bands
in $M$ so that $\lim i(\partial A_i,\mu)=0$.

We recall that there is a proper compact submanifold $\Sigma(M)$ of $M$, called the {\em characteristic submanifold}, so that 
each component of $\Sigma(M)$ is either (a) an $I$-bundle $\Sigma_0$ over a compact surface $F_0$ 
of negative Euler characteristic whose
intersection $\partial_0\Sigma_0$ with $\partial_0M$ is the associated $\partial I$-bundle or 
(b) a solid or thickened torus such that each
component of its frontier is an essential annulus. Moreover, every essential annulus or M\"obius band $A$ in $M$ is isotopic
into $\Sigma$ and if $A$ is isotopic into a component of type (a), then it is isotopic to a {\em vertical} annulus or M\"obius band,
i.e. one which is a union of fibers of the bundle.
(See Jaco-Shalen \cite{jaco-shalen} or Johannson \cite{johannson} for the general theory of characteristic
submanifolds, and \cite[Section 5]{canary-mccullough} for a discussion of the theory in the restricted setting
of hyperbolizable 3-manifolds.)
We note that there are only finitely many isotopy classes of
essential annuli or M\"obius bands in a component of type (b), so we may pass to a subsequence so
that each $A_i$ is a vertical annulus or M\"obius band in an interval bundle component $\Sigma_0$ of $\Sigma(M)$.

Let $\lambda_\infty$ be the limit of $\partial A_i$ in the space $PML(\partial_0\Sigma_0)$ of projective  
measured laminations on $\partial_0\Sigma_0$. Notice that $i(\lambda_\infty,\mu)=0$. Let $G_0$ be 
an essential  subsurface of $\partial_0\Sigma_0=\Sigma_0\cap\partial M$ which contains $\lambda_\infty$
so that $\lambda_\infty$ fills $G_0$.
Let $\iota_0:\Sigma_0\to\Sigma_0$
be the orientation reversing involution of $\Sigma_0$ which preserves each fiber of the bundle. Since 
$\partial A_i=\iota(\partial A_i)$ for all $i$, we may assume that $G_0$ is also invariant under $\iota_0$.
If $\partial G_0$ is non-empty, then the components of $\partial G_0$ bound a collection of essential annuli
and M\"obius bands which are disjoint from $\mu$, which is impossible.

It remains to consider the case when $M=\Sigma_0$  and $G_0=\partial_0\Sigma_0=\partial M$. 
Notice that $\lambda_\infty$ cannot contain a closed leaf $c$, since then $c\ \cup\iota(c)$ would bound 
an essential annulus $A$ so that $i(\partial A,\mu)=0$. Therefore,  the restriction of $\lambda_\infty$ to each component
of $\partial_0M$  is filling, so
$\lambda_\infty$  must be the support of $\mu$. 
If $M$ is untwisted, then, since $\lambda_\infty$ is invariant under the involution $\iota$,
$\nu_+(\rho)=\nu_-(\rho)$, 
which is impossible (see Thurston \cite[Prop. 9.3.7]{thurston-notes} or 
Ohshika \cite[Lem. 3.12]{ohshika-gtlimits}). If $M$ is twisted, let $\widehat M$ be the double cover which is an untwisted interval
bundle. The pre-image $\hat\lambda_\infty$ 
of $\lambda_\infty$ in $\partial\widehat M$ is invariant under the fiber-preserving involution of $\widehat M$.
This is again impossible, since $\hat\lambda_\infty$ would be the end invariant of  the associated double cover $\widehat N_\tau$
of $N_\tau$.
\end{proof}

Next we construct a sequence of laminations whose $\rho_n$-length converges to zero and such that the laminations converge 
to a measured lamination whose support is the collection of ending invariants for $\tau$.

\begin{lemma}{conv to end lam}
Suppose that  $M$ has incompressible boundary,
$\tau$ is a quasiconformally rigid representation in $AH_0(M)$,
$\{\rho_n\} \subset {\rm int}(AH_0(M))$ and 
$\{\rho_n\}$ is eventually contained in any set of the form $\mathcal U(\delta,\mathbb U,\tau)$.
Then, after possibly passing to a subsequence, there exist measured laminations $\{\mu_n\}$ such
that  
\begin{enumerate}
\item $\{\ell_{\mu_n}(\rho_n)\}$ is bounded,
\item $\{\mu_n\}$ converges to a measured lamination $\mu$ with 
support $\lambda_\tau$, and
\item the supports of $\{\mu_n\}$ Hausdorff converge to a geodesic lamination $L$ such that
every component of $L$ contains the support of exactly one component of the support of $\mu$.
\end{enumerate}
\end{lemma}

\begin{proof}
Let $S$ be a component of $\partial_0 M$.
If $W\subset S$ is an essential subsurface, $\rho\in AH(S)$, and
$L>0$, let $\CC_W(\rho,L)$ denote the set of  curves $\gamma\in\CC(W)$
whose length $\ell_\gamma(\rho) \le L$. 

Theorem 1.2 in \cite{pull-out} gives control of the set
$\CC_W(\rho,L)$ when $L$ is large enough. Specifically, there exists $L_S>0$ so that if $L\ge L_S$
then there exists $D=D(L)$ such
that, if $\CC_W(\rho,L)$ is nonempty then the Hausdorff distance
from a geodesic connecting $\pi_W(\nu_+(\rho))$ to
$\pi_W(\nu_-(\rho))$  to $\mathcal C(W,L)$ is at most $D$. In particular, if $\pi_W(\nu_+(\rho))\in\CC(W)$,
then there exists a curve in $\CC_W(\rho,L)$ within distance $D$ of $\pi_W(\nu_+(\rho))$.

To apply this to our sequence, let $W$ be a component of
$\WW_\tau$. For each $n$, let $g_n^W$ be a hyperbolic metric with geodesic boundary on $W$ and
let $f_n^W:(W,g_n^W)\to N_{\rho_n}$ be a $1$-Lipschitz map
in the homotopy class of $\rho_n|_{\pi_1(W)}$, e.g. a pleated surface.
Since $\ell_{\partial W}(\rho_n)\to 0$, there exists $L_W>L_S$, so that, for all large enough $n$,
$g_n^W$ contains a non-peripheral curve of length at most $L_W$, so
$\CC_W(\rho_n,L_W)$ is nonempty. Applying the previous paragraph to $\rho_n^S=\rho_n|_{\pi_1(S)}$,
we obtain a curve $c^n_W$ in $W$ with
$d_{\CC(W)}(c^n_W,\pi_W(\nu(\rho_n))) \le D_W$ and
$\ell_{c^n_W}(\rho_n)\le L_W$, for some uniform $D_W=D(L_W)$.

Since $\{\pi_W(\nu(\rho_n))\}$ converges to $\lambda_W\in\partial_\infty\mathcal C(W)$,
we may choose a sequence $r_n^W\to\infty$ so that any subsequence of 
$\{c_n^W/r_n^W\}\subset \ML(W)$ has a subsequence which converges to a measured lamination with support $\lambda_W$,
where $\ML(W)$ is the space of measured laminations on $W$, see Klarreich \cite[Thm. 1.4]{klarreich}. Moreover, any
subsequence of $\{c_n^W\}$ has a subsequence which Hausdorff converges to a connected geodesic lamination supported 
in the interior of $W$.
If 
$$\mu_n= p_\tau\cup \bigcup_{W\in\mathcal W_\tau} \frac{c_n^W}{r_n^W}$$
then $\lim\ell_{\mu_n}(\rho_n) = 0$ so we may pass to a  subsequence of $\{\mu_n\}$  which
converging to a measured lamination $\mu$ whose support is 
$\lambda_\tau$. After further passage to a subsequence, we may assume that the supports of $\{\mu_n\}$ Hausdorff converge to
a geodesic lamination $L$. By construction, $p_\tau$ is contained in $L$  and if $W\in W_\tau$, then the restriction of $L$
to $W$ is supported on the interior of $W$ (since it is a limit of non-peripheral curves on $W$.) Therefore, every component of $L$ contains
in its support exactly one component of the support of $\mu$.
\end{proof}

We now combine the preceding results and topological arguments to establish Proposition \ref{nbhd system}.

\begin{proof}[Proof of Proposition \ref{nbhd system}]
It remains to prove that  if $\{\rho_n\}$ is eventually contained in every set of the form $\mathcal U(\delta, \mathbb U, \tau)$,
then $\{\rho_n\}$ converges to $\tau$ (or accumulates on $\{\tau,\bar\tau\}$ when $M$ is an interval bundle and
$\mathcal W_\tau$ is empty). It suffices to prove that every such sequence has a subsequence that converges to $\tau$
(or into $\{\tau,\bar\tau\}$ if $M$ is an interval bundle and
$\mathcal W_\tau$ is empty).

If $\{\rho_n\}$ is eventually contained in every set of the form $\mathcal U(\delta, \mathbb U, \tau)$,
Lemmas \ref{end lam doubly} and \ref{conv to end lam} imply that,
after possibly passing to a subsequence, there exist
measured laminations $\{\mu_n\}$
so that $\ell_{\mu_n}(\rho_n)\to 0$, $\{\mu_n\}$ converges to a doubly incompressible measured lamination $\mu$
with support $\lambda_\tau$, and the supports of $\{\mu_n\}$ converge to a measured lamination $L$ with the property
that every component of $L$ contains exactly one component of the support of $\mu$. Since there are no laminations 
on $\partial_0M$ with
support disjoint from $\lambda_\tau$ and every component of $\lambda_\tau$ is minimal, every component of $L$ contains a unique minimal  lamination which is a
component of $\lambda_\tau$.
Theorem \ref{doublyincomp} then implies that, passing to a further subsequence, 
$\{\rho_n\}$ converges to some $\rho\in AH(M)$.

We must show that $\rho = \tau$ (or that $\rho\in\{\tau,\bar\tau\}$ if $M$ is an interval bundle and
$\mathcal W_\tau$ is empty).
To do so, we will show that $N_\rho$ has the same marked homeomorphism
type as $N_\tau$ and that both manifolds have the same end-invariants, 
and then apply the Ending Lamination Theorem \cite{ELC2}. 

Let $M_\rho$ be a relative compact core for $N_\rho^0$ and let $P_\rho=\partial M_\rho\cap\partial N_\rho^0$.
Similarly, let $M_\tau$ be a relative compact core for $N_\tau^0$ and let $P_\tau=\partial M_\tau\cap\partial N_\tau^0$.
We claim that there exists a homotopy-equivalence $j:M_\tau\to M_\rho$, 
in the homotopy class of $\rho\circ\tau^{-1}$, such that
\begin{enumerate}
  \item $j$ takes $P_\tau$ homeomorphically to a subcollection
    $\hhat P$ of $P_\rho$,
\item
there exists a submanifold $Z$ of $\partial M_\tau \setminus P_\tau$ which consists of a compact core 
of each component of $\partial M_\tau \setminus P_\tau$ which is not a thrice-punctured sphere, such
that $j$ restricts to an orientation-preserving embedding of $Z$ into $\partial M_\rho \setminus \hhat P$
and each component of $j(Z)$ is a compact core for a component of $\partial M_\rho \setminus P_\rho$.
\end{enumerate}
In this situation, Proposition 8.1 of Canary-Hersonsky
\cite{canary-hersonsky} asserts
that $j$ is homotopic, as a map of pairs, 
to a homeomorphism of pairs 
$$J:(M_\tau,P_\tau)\to (M_\rho,\hhat P)$$
which agrees with $j$ on $Z$.
(Canary and Hersonsky's result is a nearly immediate consequence of Johannson's Classification Theorem
\cite{johannson}.)
Furthermore, we will see that  every component of $j(Z)$ is a compact core for a component of $M_\rho-\hhat P$ 
which bounds a geometrically infinite end of $N_\rho^0$. It follows that
$\hhat P = P_\rho$.

Since $M_\tau$ and $M_\rho$  are aspherical and $\rho\circ\tau^{-1}$ gives an isomorphism of 
$\pi_1(M_\tau)=\tau(\pi_1(M))$ to $\pi_1(M_\rho)=\rho(\pi_1(M))$, there is a homotopy equivalence 
$j:M_\tau\to M_\rho$ in the homotopy class of $\rho\circ\tau^{-1}$.
To establish condition (1), note that
$\ell_{\rho}(p_\tau)=0$ . Hence we may choose $j$ to take all  annular components of $P_\tau$ to annular
components of $P_\rho$. Moreover, if $T$ is a toroidal component of $P_\tau$, then $j(T)$ is homotopic to
a toroidal component of $P_\rho$, so we may assume that $j(T)\subset P_\rho$ in this case as well.
Since the restriction of $j$ to any component of $P_\tau$ is a homotopy equivalence to the image component of
$P_\rho$, we may assume that the restriction of $j$ to each component of $P_\tau$ is a homeomorphism
onto a component of $P_\rho$.

In order to establish condition (2), we will show that geometrically infinite ends of $N^0_\tau$ are associated
to geometrically infinite ends of $N^0_\rho$ with the same ending lamination.
If $W\in\mathcal W_\tau$ and $S$ is the component of $\partial_0M$ containing $W$,
then $\{\pi_W(\nu(\rho_n))\}$ converges to $\lambda_W$, so
Theorem \ref{endlams in limit} implies that $\rho^S=\rho|_{\pi_1(S)}$ has an upward-pointing
geometrically infinite end with support $W$ and ending lamination  $\lambda_W$.
Condition (2) then follows from the following more general claim.

\begin{lemma}{no finite cover}
Suppose that $M$ is a compact, orientable, hyperbolizable \hbox{3-manifold} with non-empty incompressible boundary
and that $W$ is an essential subsurface of a boundary component $S$ of $M$ which is not an annulus
or a pair of pants.
If $\rho\in AH(M)$, $M_\rho$ is a relative compact core for $N_\rho^0$ and $\rho^S$ has an outward-pointing 
geometrically infinite end with base surface $W$,
then there exists a homotopy equivalence $g:M\to M_\rho$ which
restricts to an orientation-preserving homeomorphism
of $W$ to a subsurface  $g(W)$ in $\partial M_\rho$ which bounds a geometrically infinite end of $N_\rho^0$
\end{lemma}

\begin{proof}
The Covering Theorem (\cite{canary-cover}  and \cite[Thm. 9.2.2]{thurston-notes}) implies that the covering map
$\pi:N_{\rho^S}\to N_\rho$ restricts to a finite-to-one cover from
a neighborhood $U$ of the end  of $N_{\rho^S}^0$ associated to $W$ to a neighborhood $\widehat U$ of a
geometrically infinite end of $N^0_\rho$.
Let $W'$ be the component of $\partial M_\rho\setminus P_\rho$ associated to the end of $N_\rho^0$ with neighborhood 
$\widehat U$.  Since $U$ has a subneighborhood homeomorphic to $W\times (0,\infty)$ which maps into 
a subneighborhood of $\widehat U$ which is homeomorphic to $W'\times (0,\infty)$, we see that there exists a
homotopy equivalence $g:M\to M_\rho$ so that 
$g|_W$ is a proper \hbox{$\pi_1$-injective}  orientation-preserving map onto $W'$,
so we may assume $g|_W$ is an orientation-preserving finite cover of
its image (see \cite[Prop. 3.3.]{johannson}).

If $g|_W$ is not a homeomorphism, then there exists a curve $\alpha$ on $W$ which is indivisible
in $\pi_1(W)$ (i.e. generates a maximal cyclic subgroup), but whose image $g(\alpha)$ is divisible in $\pi_1(W')$
(i.e. does not generate a maximal cyclic subgroup of $\pi_1(W')$).
It follows that $\alpha$ is a peripheral curve which is divisible in $M$, but not in $\partial M$.
A result of Johannson \cite[Lem. 32.1]{johannson} implies that  either 
(i) $\alpha$ bounds an immersed essential M\"obius band $B$ in an interval bundle component $\Sigma_0$
(with base surface $F_0$) of $\Sigma(M_\tau)$,
or (ii) $\alpha$ is homotopic to the core curve of an annulus in the frontier of a solid torus component $V$ of $\Sigma(M_\tau)$. 
In case (i), Johannson's Classification Theorem \cite[Thm. 24.2]{johannson} implies that  $g$ is homotopic to a 
homotopy equivalence which restricts to a homeomorphism of $\Sigma_0$ to  an interval bundle component $g(\Sigma_0)$  of
$\Sigma(M_\rho)$, so $g(\alpha)$ bounds an immersed essential M\"obius band in $M_\rho$. 
In case (ii), Johannson's Classification Theorem  implies that $g$ is homotopic to a homotopy
equivalence which takes $V$ to a solid torus component $g(V)$ of $\Sigma(M_\rho)$ and restricts to a homeomorphism
between the frontier of $V$ in $M$ and the frontier of $g(V)$ in $M_\rho$, so $g(\alpha)$ is 
is homotopic to the core curve of an  annulus in the frontier of $g(V)$.
In either case, $g(\alpha)$ cannot be divisible in $\pi_1(W')\subset \pi_1(M_\rho)$, so we have again
achieved a contradiction.
\end{proof}

Since we have adjusted $j$ to satisfy conditions (1) and (2) above, 
we may apply Proposition 8.1 of \cite{canary-hersonsky} to upgrade $j$ to  a homeomorphism
of pairs $J:(M_\tau,P_\tau)\to(M_\rho,P_\rho)$.

If $J$ is orientation-preserving, then the Ending Lamination
Theorem tells us that $\rho=\tau$ in $AH(M)$ and we are done.
In particular, we are done if $\mathcal W_\tau$ is non-empty, since the restriction of $J$ to
each element of $\mathcal W_\tau$ is an orientation-preserving homeomorphism.

If $J$ is orientation-reversing, then the Ending Lamination
Theorem tells us that $\rho=\bar\tau$ in $AH(M)$. Since we may take $M_{\bar\tau}$ to be $M_\tau$ with
the opposite orientation, $J$ is an orientation-reversing involution of $M_\tau$ which is homotopic
to the identity. It follows from the lemma below that $M_\tau$, and therefore $M$, is an interval bundle.
So, since $M$ is an interval bundle, $\mathcal W_\tau$ is empty,  and $\rho$ lies in $\{\tau,\bar\tau\}$,
we are again done.

\begin{lemma}{}
If $M$ is a compact, orientable, hyperbolizable 3-manifold with non-empty incompressible boundary, and 
$g:M\to M$ is an orientation-reversing involution which is homotopic to the identity,
then $M$ is an interval bundle.
\end{lemma}

\begin{proof}
Let $S$ be a boundary component of $M$ and let 
$H:S\times [0,1]\to M$ be the restriction to $S$ of the homotopy of $g$ to the identity map.
Notice that  $H$ is not properly homotopic to a map $\hat H$ with image in $\partial M$, 
since then $\hat H$  would be a homotopy of an orientation-reversing involution of $S$ to an orientation-preserving
homeomorphism of $S$. Then, a generalization of Waldhausen's theorem,
see Hempel \cite[Theorem 13.6]{hempel}, implies that $H$ is properly homotopic to a covering map. 
Therefore,  $M$ must be an interval bundle, see Hempel \cite[Thm 10.6]{hempel}.
\end{proof}
\end{proof}

\section{Bounds on the skinning map}
\label{skinning}

Thurston's skinning map records the geometry, i.e. end invariants, of the ``inward-pointing'' ends of the covers
of a hyperbolic 3-manifold  associated to its peripheral subgroups. 
In this section, we use the results of \cite{BBCL} to  show that certain aspects of the geometry of these inward-pointing ends
is uniformly bounded as one approaches a quasiconformally rigid hyperbolic 3-manifold.

Let $M$ be a compact, orientable hyperbolizable 3-manifold with incompressible boundary,
which is not an untwisted interval bundle. If \hbox{$\rho\in AH_0(M)$} and $S$ is a component of
$\partial_0 M$, then $\rho^S=\rho|_{\pi_1(S)}$ is a Kleinian surface group with end invariants $(\nu_+(\rho^S),\nu_-(\rho^S))$.
We assume that we have chosen an orientation on $S$ which is consistent with the orientation on $M$,
so that $\nu_+(\rho^S)$ is the restriction $\nu(\rho)|_S$ of $\nu(\rho)$ to $S$.
One then defines a map $\sigma_S$, with image in the set
$\mathcal E(S)$ of ending invariants on $S$, by setting
$\sigma_S(\rho)=\nu_-(\rho^S)$. 
The {\em skinning map} is the product map
$$\sigma_M =\prod\sigma_S:AH_0(M)\to \mathcal E(\partial_0M)$$ 
where the product is taken over all components of $\partial_0 M$ and 
$\mathcal E(\partial_0 M)$ is the product of the $\mathcal E(S)$. 

\medskip
\noindent
{\bf Remark:} In this paper we only actually apply the skinning map to representations in ${\rm int}(AH_0(M))$.
If $\partial M$ does not contain tori, then the image of $\rho\in {\rm int}(AH_0(M))$ under $\sigma_M$ will lie in the 
Teichm\"uller space $\Teich(\partial M)$. If $\partial M$ does contain tori, let $t$ be the collection of curves 
on $\partial_0M$ which are homotopic into a toroidal component of 
$\partial M$. If $\rho\in {\rm int}(AH_0(M))$, then $\sigma_M(\rho)$ consists of the parabolic locus $t$ and a conformal structure on
$\partial_0 M - t$.
(See \cite[Sec. 2]{BBCM-windows} for a more detailed discussion
of the skinning map in the setting of 3-manifolds with incompressible boundary.)

\begin{proposition}{preliminary bounds} 
Let $M$ be a compact, orientable, hyperbolizable 3-manifold with incompressible boundary  
and let $\mu$ be a complete marking on $\partial_0 M$.
If $\tau\in AH_0(M)$ is quasiconformally rigid, then
there exists a neighborhood $\UU$ of $\tau$ such that for all $\rho\in\UU$
and $\alpha\in p_\tau$,
$$\mm_\alpha(\sigma_M(\rho),\mu) < R.$$
\end{proposition}

\begin{proof}
If the proposition fails, then there exists a sequence $\{\rho_n\}$ in
${\rm int}(AH_0(M))$ converging to $\tau$
and $\alpha\in p_\tau$
such that
$\mm_\alpha(\sigma_M(\rho_n),\mu)\to \infty$.
Let $S$ be the component of $\partial_0M$ containing $\alpha$, so 
$\mm_\alpha(\sigma_M(\rho_n),\mu)=\mm_\alpha(\sigma_S(\rho_n^S),\mu|_S)$ and
$\{\rho_n^S\}$  converges to $\tau^S=\tau|_{\pi_1(S)}$.

In order to apply Theorem \ref{BBCL result}, we need to find a
geometric limit $\hat\Gamma^S$ for $\{\rho_n(\pi_1(S))\}$ so that a
compact core of $N_{\tau^S}$
embeds in $\widehat N^S=\mathbb H^3/\hat\Gamma_S$.
Pass first to a subsequence so that $\{\rho_n(\pi_1(M))\}$ converges geometrically to $\hat\Gamma$
and, by passing to a further subsequence, so
that $\{\rho_n(\pi_1(S))\}$ converges geometrically to $\hat\Gamma^S\subset\hat\Gamma$. 
It follows from Proposition \ref{convexembed} that there exists a compact core $C_\tau$ for $N_\tau$ which
embeds, under the natural covering map,
in $\widehat N= \mathbb H^3/\hat\Gamma$. 

Since $\tau\in AH_0(M)$, there exists a homeomorphism $h_\tau:M\to M_\tau$ so that $(h_\tau)_*$ is  conjugate to $\tau$,
where $M_\tau$ is a relative compact core for $N_\tau^0$. Moreover, by a result of McCullough, Miller and Swarup \cite{MMS},
there exists a homeomorphism $g:M_\tau\to C_\tau$ so that $f=g\circ h_\tau$ is in the homotopy class of $\tau$.
Since $f(S)$ embeds in $\hat N$, it admits a compact regular neighborhood $X$ which embeds in $\widehat N$.
Then, $X$ lifts to a compact core for $N_{\tau^S}$ which embeds in $\widehat N^S=\mathbb H^3/\hat\Gamma^S$.

Notice that $\sigma_S(\rho_n)=\nu(\rho_n^S)_-$ for all $n$.
Since $\alpha$ is an upward-pointing parabolic for $\tau|_{\pi_1(S)}$, 
Theorem \ref{BBCL result} then implies that if $\mu$ is a complete marking on $S$, then
$\{\mm_\alpha(\sigma_S(\rho_n),\mu|_S)\}$ is eventually bounded, which provides
a contradiction.
\end{proof}

\section{Proofs of main results}

Suppose that $\tau$ is quasiconformally rigid. We may assume,
by precomposing by an element of ${\rm Out}(\pi_1(M))$ and 
by replacing $M$ by $M_\tau$, that $\tau\in AH_0(M)$.
Theorem \ref{no bump} then implies that $\tau$ lies in the boundary of ${\rm int}(AH_0(M))$ and does not lie
in the boundary of any other component of the interior of $AH(M)$. Therefore, Theorem
\ref{nonselfbump}, which asserts that no component of ${\rm int}(AH(M))$ self-bumps at $\tau$, follows
immediately from the following result.

\begin{proposition}{joining nearby points}
Let $\tau$ be a quasiconformally rigid representation in $AH_0(M)$. Given any
open neighborhood $V$ of $\tau$, there exists a sub-neighborhood $\hat V\subset V$ of $\tau$ so that
any two points in $\hat V\cap {\rm int}(AH_0(M))$ can be joined by a path in $V\cap {\rm int}(AH_0(M))$. 
\end{proposition}

Notice that Theorem \ref{main}, which asserts that $AH(M)$ is locally connected at $\tau$,  follows  from the
facts that no two components of ${\rm int}(AH(M))$ bump at $\tau$ (Theorem \ref{no bump}), no component of
${\rm int}(AH(M))$ self-bumps at $\tau$, and that $AH(M)$ is the closure of its interior \cite{ELC2}.

\medskip

We will make crucial use of results from our earlier work \cite{nobumping} which analyzed the
relationship between Fenchel-Nielsen coordinates on Teichm\"uller space and subsurface projections.
The first result allows
us to pinch curves  in the conformal boundary while controlling complementary subsurface projections.

\begin{lemma}{twist and shrink}{\rm (\cite[Lemma 6.1]{nobumping})}
Given a (possibly disconnected) surface $S$ and constants $K$ and $\delta>0$, there exists $c=c(S)$
and $h=h(\delta,K,S)$ such that
if $X\in \Teich(S)$, $\mu$ is a complete marking of $S$ 
and $\ba$ is a curve system on $S$  such that
$$
\mm_{\alpha}(X,\mu) > h  (\delta,K,S)
$$
for each component $\alpha$ of $\ba$, then there exists a path
$\{X_t:t\in[0,1]\}$ in $\Teich(S)$ with $X_0=X$ such that 
\begin{enumerate}
\item 
$l_{\alpha}(X_1) < \delta/2$ for each $\alpha\in\ba$,
\item  
$\mm_{\alpha}(X_t,\mu) > K$ for each $\alpha\in\ba$ and each $t\in[0,1]$, and
\item 
$\diam(\pi_W(\{X_t:t\in[0,1]\})) < c$, for any subsurface $W$ disjoint from $\ba$.
\end{enumerate}
\end{lemma}

The second result allows one to join surfaces where a multicurve is  short in a controlled manner.

\begin{lemma}{W connected}{\rm (\cite[Lemma 5.11]{nobumping})} 
Given a collection $\ba$ of disjoint, non-parallel, essential simple closed curves on a (possibly disconnected) surface $S$, 
let  $\mathcal W$ be the collection of components of $S\setminus\ba$ 
which are not thrice-punctured spheres. If $\ep < \ep_0$, $\{\lambda_W\}_{W\in\mathcal W}$ is a collection of filling laminations on
components of $\mathcal W$ and $\mathbb U=\{U_W\}_{W\in\mathcal W}$ where each $U_W$ is a neighborhood of 
$\lambda_W$ in $\mathcal{C}(W)$, then
there exist neighborhoods $\mathbb U'=\{U'_W\subset U_W\}_{W\in\mathcal W}$ of $\lambda_W$ in $\mathcal{C}(W)$,
such that if $Y_0,Y_1\in \mathcal T(S)$ and $\ell_\alpha(Y_i)<\delta$ for all $\alpha \in\ba$ and 
$\pi_W(Y_i)\in\mathcal U'$ for all $W\in\mathcal W$ and $i=0,1$, then there exists a path $\{Y_t\}_{t\in [0,1]}$ joining $Y_0$ to $Y_1$
such that 
\begin{enumerate}
\item
$\ell_\alpha(Y_t)<\delta$ for all $\alpha \in\ba$  and $t\in[0,1]$, and 
\item
$\pi_W(Y_t)\in\mathcal U$ for all $W\in\mathcal W$  and $t\in[0,1]$.
\end{enumerate}
\end{lemma}

\noindent
{\bf Remark:} In Lemma 6.1 in \cite{nobumping} the basepoint $\mu$ is a conformal structure on $S$, 
while here it is a complete marking. However, as we are only interested in the projection to the curve complex 
and for any conformal structure there is a complete marking 
that has coarsely the same image in the projection to any curve complex 
the statements are equivalent. We are using a complete marking here to match with Proposition \ref{preliminary bounds}.

\medskip\noindent
\begin{proof}[Proof of Proposition \ref{joining nearby points}]
Let $V'$ be an open neighborhood of $\bar\tau$ in $AH(M)$  so that $V$ and $V'$ have disjoint closures.
If $\tau=\bar\tau$, let
$V=V'$. (Note, we only need $V'$ in the case of interval bundles).
Proposition \ref{nbhd system} implies that there exists $\delta>0$ and
$\mathbb U=\{U_W \}_{W\in \mathcal W_\tau}$  so that each $U_W$
is a neighborhood of $\lambda_W\in\partial_\infty\mathcal C(W)$ and 
$$\mathcal U(\delta, \mathbb U,\tau)\subset V\cup V'.$$

Lemma \ref{W connected} gives a collection $\mathbb U'=\{ U'_W\}_{W\in\mathcal W_\tau}$ 
of sub-neighborhoods of $\mathbb U$ so that if we let $\mathcal T(\delta/2,\mathbb U')$
be the set of surfaces $Y\in\mathcal T(\partial_0M)$ such that 
\begin{enumerate}[(A)]
\item
$\ell_\alpha(Y)<\delta/2$  for all $\alpha\in p_\tau$, and 
\item
$\pi_W(Y)\in U_W'$ for all $W\in\mathcal W_\tau$,
\end{enumerate}
then any $Y_0,Y_1\in\mathcal T(\delta/2,\mathbb U')$ can be connected by a path $\{Y_t\ |\ t\in [0,1]\}$ so that 
\hbox{$\ell_\alpha(Y_t)<\delta/2$} and $\pi_W(Y_t)\in U_W$ for all $t\in [0,1]$ and all $W\in\mathcal W_\tau$.
Let $\{\rho_t\}_{t\in[0,1]}$ be the path in ${\rm int}(AH_0(M))$ so that $\nu(\rho_t)=Y_t$.
Bers \cite[Thm. 3]{bers-slice} proved that if $\alpha$ is a simple closed curve on $\partial_0M$, $M$ has
incompressible boundary and $\sigma\in{\rm int}(AH_0(M))$, then $\ell_\alpha(\sigma)\le 2\ell_\alpha(\nu(\sigma))$, so this path is
entirely contained in $\mathcal U(\delta, \mathbb U, \tau)$. To summarize, we have shown that any two points  
$\rho_0,\rho_1\in{\rm int}(AH_0(M))$
such that $\nu(\rho_0),\nu(\rho_1)\in \mathcal T(\delta/2,\mathbb U')$ may be joined by a path
entirely contained in $\mathcal U(\delta, \mathbb U, \tau)$

We will complete the proof by finding $\epsilon>0$ and  a collection \hbox{$\widehat{\mathbb U} = \{{\hat U}_W\}_{W\in \mathcal W_\tau}$}
of  sub-neighborhoods of $\mathbb U$, so that if \hbox{$\rho \in \mathcal U(\epsilon, \widehat{\mathbb U}, \tau)$},
then $\rho$ can be connected to a representation $\hat\rho$ so that $\nu(\hat\rho)\in\mathcal T(\delta/2,\mathbb U')$ by a 
path entirely contained in $\mathcal U(\delta, \mathbb U, \tau)$. It then follows that any two points in
$\mathcal U(\epsilon, \widehat{\mathbb U}, \tau)$ can be joined by a path in  $\mathcal U(\delta, \mathbb U, \tau)$.

Proposition \ref{nbhd system} will then imply that there exists a neighborhood $\hat V\subset V$ of $\tau$ so that
$\hat V\cap {\rm int}(AH_0(M))$ is contained in $\mathcal U(\epsilon, \widehat{\mathbb U}, \tau)$. It follows that any two points in 
$\hat V\cap {\rm int}(AH_0(M))$ can be joined by a path in $\mathcal U(\delta, \mathbb U,\tau)\subset V\cup V'$.
Since $V$ and $V'$ have disjoint closures, the path must be entirely contained in $V$.

\bigskip

We fix a complete marking $\mu$ on $\partial_0M$.
Proposition \ref{preliminary bounds} provides a neighborhood $ \mathcal U(\delta_0,\mathbb U_0,\tau)$
of $\tau$, where $\mathbb U_0=\{(U_0)_W\}_{W\in\mathcal W_\tau}$, and $R>0$,
such that 
$$\mm_\alpha(\sigma_M(\rho),\mu)<R.$$
if $\rho\in \mathcal U(\delta_0,\mathbb U_0,\tau)$.
We may assume, without loss of generality, that $\delta<\delta_0$ and $\mathbb U\subset\mathbb U_0$.

We will apply Lemma \ref{twist and shrink} to find $\epsilon>0$ and $\widehat{\mathbb U}$ so that any
$\rho_0\in\mathcal U(\epsilon,\widehat{\mathbb U},\tau)$
may be joined to a representation $\hat\rho$ so that  $\nu(\rho_1)\in\mathcal T(\delta/2,\mathbb U')$  by a path $\{\rho_t\}$
in  $\mathcal U(\delta, \mathbb U, \tau)$. 
It is easy to use part (3) of Lemma \ref{twist and shrink} to choose the sub-neighborhoods $\widehat{\mathbb U}$
so that the projection of the path stays in $\mathbb U$. It is more difficult to ensure that each curve $\alpha$ in $p_\tau$ 
has length less than $\delta$ on the entire path. Here we use part (2) of Lemma \ref{twist and shrink}
to show that $\mm_\alpha(\nu(\rho_t), \nu(\sigma_M(\rho_t)))$ is large, which in turn, by Theorem \ref{KGCC results},
will imply that $\alpha$ is short.

Theorem \ref{KGCC results} provides $K>0$, so that if $\gamma$ is a simple closed curve in $\partial_0M$, 
$\rho\in AH_0(M)$ and $\mm_\gamma(\nu(\rho),\sigma_M(\rho))>K$, then 
$$\ell_\gamma(\rho)<\delta/2.$$

Let $c= c(\partial_0M)$ and $h=h(\delta,K+R,\partial_0M)$ be the constants given by Lemma \ref{twist and shrink}.
Theorem \ref{KGCC results} implies that there exists
$\ep>0$ so that if $\gamma$ is a simple closed curve in $\partial_0 M$ and $\ell_\gamma(\rho)<\ep$,
then 
$$\mm_\gamma(\nu(\rho),\sigma_M(\rho))>h+R.$$
Let $\widehat{\mathbb U}=\{ \hat U_W\}_{W\in\mathcal W_\tau}$ be a collection of neighborhoods 
of the ending laminations of $\tau$ so that, for each $W\in W_\tau$,
the neighborhood of $\hat{U}_W$ of radius $c$ in $\mathcal C(W)$ is contained in $U_W'$.

If $\rho \in\mathcal U(\epsilon, \widehat{\mathbb U}, \tau)$, then $\ell_\alpha(\rho) < \epsilon$,
for all $\alpha \in p_\tau$, so
$$\mm_\gamma(\nu(\rho),\sigma_M(\rho))>h+R,$$
which implies
that $\mm_\gamma(\nu(\rho),\mu)>h$, since $\mm_\gamma(\mu,\sigma_M(\rho))<R$. 
Lemma \ref{twist and shrink} then implies that there exists a path
$\{X_t\ |\ t\in [0,1]\}$ in $\Teich(\partial_0 M)$ with $X_0 = \nu(\rho)$ so that
\begin{enumerate}
\item 
$l_{\alpha}(X_1) < \delta/2$ for each $\alpha\in p_\tau$,
\item  
$\mm_{\alpha}(X_t,\mu) > K+R$ for each $\alpha\in p_\tau$ and each $t\in[0,1]$, and
\item 
$\diam(\pi_W(\{X_t:t\in[0,1]\})) < c$  for all  $W\in\mathcal W_\tau$.
\end{enumerate}
Let $\{\rho_t\ |\ t\in [0,1]\}$  be the path  in ${\rm int}(AH_0(M))$ with $\nu(\rho_t) = X_t$ for all $t$. 
It only remains to check that $\{\rho_t\}\subset \mathcal U (\delta, \mathbb U, \tau)$ for all $t$  and that 
\hbox{$\nu(\rho_1)\in\mathcal T(\delta/2,\mathbb U')$}.

Property (1) implies that $\ell_\alpha(\nu(\rho_1)) < \delta/2$ for all $\alpha\in p_\tau$.
Property (3) implies that if $W\in\mathcal W_\tau$, then
$$\diam(\pi_W(\{X_t:t\in[0,1]\})) < c.$$
It follows from the definition of $\hat U_W$ and
the fact  that \hbox{$\pi_W(X_0)=\pi_{W}(\rho)\in\hat{U}_W$}, that
$$\pi_{W}(\rho_t)=\pi_W(X_t)\in U_W'\subset U_W$$
for all $t\in [0,1]$ and all $W\in\mathcal W_\tau$.
In particular, \hbox{$\nu(\rho_1)\in\mathcal T(\delta/2,\mathbb U')$}.

In order to verify that  $\rho_t\in \mathcal U(\delta,\mathbb U,\tau)$ for all $t\in [0,1]$, 
it remains to check that $\ell_\alpha(\rho_t)<\delta$ for all $t\in [0,1]$.
If this is not the case, there exists 
$s \in [0,1]$ and $\alpha\in p_\tau$ so that $\ell_{\alpha}(\rho_s) = \delta<\delta_0$. 
Since 
$\rho_s\in \mathcal U(\delta_0,\mathbb U_0,\tau)$, we know that
$\mm_{\alpha}(\sigma_M(\rho_s),\mu)<R$. Therefore, by applying Property (2) above and the triangle inequality for
$\mm_\alpha$, we see that
\begin{eqnarray*}
\mm_{\alpha}(\nu(\rho_s),\sigma_M(\rho_s)) &= & \mm_{\alpha}(X_s,\sigma_M(\rho_s))\\
& \ge & \mm_{\alpha}(X_s,\mu)-\mm_{\alpha}(\sigma_M(\rho_s),\mu)\\
 & > & K+R-R=K.\\
 \end{eqnarray*}
So, by our assumptions on $K$, $\ell_{\alpha}(\rho_s)<\delta/2$,  which is a contradiction.

This completes the proof that $\rho_t \in \mathcal U(\delta, \mathbb U,\tau)$ for all $t \in [0,1]$, and hence 
Proposition \ref{joining nearby points}.
\end{proof}


\begin{thebibliography}{100}
{\footnotesize

 
\bibitem{ACpages} J. Anderson and R.  Canary, ``Algebraic limits of 
Kleinian groups which rearrange the pages of a book,'' 
\emph{Invent. Math.} \textbf{126} (1996), 205--214.

\bibitem{ACCS}
J.~Anderson, R.~Canary, M.~Culler, and P.~Shalen, ``Free {K}leinian groups
  and volumes of hyperbolic $3$-manifolds,'' {\em J. Diff. Geom.} \textbf{43}(1996),  738--782.

\bibitem{ACM}
J.~Anderson, R.~Canary, and D.~McCullough, ``The topology of deformation
spaces of {K}leinian groups,'' {\em Annals of Math.} \textbf{152} (2000),
693--741.

\bibitem{anderson-lecuire} J. Anderson and C. Lecuire,  ``Strong convergence of Kleinian groups: 
the cracked eggshell,'' {\em Comm. Math. Helv.} {\bf 88} (2013), 813--857.
  

\bibitem{bers-slice} L. Bers, ``On boundaries of Teichm\"uller spaces and
Kleinian groups I,'' {\em Annals of Math.} {\bf 91}(1970), 570--600.

\bibitem{bers-spaces} L. Bers, ``Spaces of Kleinian groups,'' in {\em Maryland Conference in Several Complex Variables
I.} Springer-Verlag Lecture Notes in Math, No. 155 (1970), 9--34.

\bibitem{bers-survey} L. Bers, ``On moduli of Kleinian groups,''
\emph{Russian Math. Surveys} \textbf{29} (1974), 88--102.

\bibitem{bonahon} F. Bonahon, ``Bouts des vari\'et\'es hyperboliques de dimension 3,''
\emph{Annals of Math.} \textbf{124} (1986), 71--158.

\bibitem{bonahon-otal} F. Bonahon and J.P. Otal, ``Laminations mesur\'ees de plissage des vari\'et\'es hyperboliques
de dimension 3,'' {\em Annals of Math.} {\bf 160}(2004), 1013--1055.

\bibitem{bowditch} B. Bowditch, ``The ending lamination theorem,'' preprint, available at:
{\texttt http://homepages.warwick.ac.uk/~masgak/papers/elt.pdf}

\bibitem{brock-invariants} J. Brock, ``Boundaries of Teichm\"uller spaces and end-invariants for hyperbolic \hbox{$3$-manifolds},''
{\em Duke Math. J.} {\bf 106}(2001),  527--552. 

\bibitem{brock-bromberg} J. Brock and K. Bromberg, ``On the density of
geometrically finite Kleinian groups,''  {\em Acta Math.}  {\bf 192}(2004), 33--93.

\bibitem{BBCL} J. Brock, K. Bromberg, R. Canary and C. Lecuire,
``Convergence and divergence of Kleinian surface groups,'' {\em J. Topology}
{\bf 8}(2015), 811--841.

\bibitem{nobumping} J. Brock, K. Bromberg, R. Canary and Y. Minsky,
``Local topology in deformation spaces of hyperbolic 3-manifolds,''
{\em Geom. Top.} {\bf 15}(2011), 1169--1224.

\bibitem{pull-out} J. Brock, K. Bromberg, R.  Canary and Y. Minsky,
``Convergence properties of ending invariants,'' {\em Geom. Top.} 
{\bf 17}(2013), 2877--2922.

\bibitem{BBCM-windows} J. Brock, K. Bromberg, R. Canary and Y. Minsky,
``Windows, cores and skinning maps,'' preprint, available at:
{\texttt https://arxiv.org/abs/1601.05482}

\bibitem{ELC2} J. Brock, R. Canary and Y. Minsky,
``The classification of Kleinian
surface groups II: the ending lamination conjecture,'' 
{\em Annals of Math.}  {\bf 176}(2012), 1--149.


\bibitem{bromberg-density} K. Bromberg, ``Projective structures with
degenerate holonomy and the Bers density conjecture,'' {\em Annals of Math.}
{\bf 166}(2007), 77--93.

\bibitem{bromberg-PT} K. Bromberg, ``The space of
Kleinian punctured torus groups is not locally connected,''  {\em Duke Math. J.}
{\bf 156}(2011), 387--427.

\bibitem{bromberg-holt} K. Bromberg and J. Holt, ``Self-bumping of deformation
spaces of hyperbolic \hbox{$3$-manifolds},'' \emph{J. Diff. Geom.}
\textbf{57}(2001), 47--65.

\bibitem{canary-cover} R. Canary,  ``A covering theorem for hyperbolic 3-manifolds and its
applications,'' {\em Topology}  \textbf{35}(1996), 751--778.


\bibitem{canary-bumponomics} R. Canary, ``Introductory bumponomics: the topology of deformation spaces
of hyperbolic 3-manifolds,'' in {\em Teichm\"uller Theory and Moduli Problem}, ed. by I. Biswas, R. Kulkarni 
and S. Mitra, Ramanujan Mathematical Society, 2010, 131--150,

\bibitem{canary-hersonsky} R. Canary and S. Hersonsky, ``Ubiquity of geometric 
finiteness in boundaries of deformation spaces of hyperbolic 3-manifolds,''
{\em Amer. J. Math.} {\bf 126}(2004), 1193--1220.

\bibitem{canary-mccullough} R. Canary and D. McCullough,
{\em Homotopy equivalences of 3-manifolds and deformation theory of
Kleinian groups,}   Mem. Amer. Math. Soc. {\bf 172}(2004), no. 812.


\bibitem{hamenstadt} U. Hamenst\"adt, ``Train tracks and the Gromov boundary of the complex of curves,'' 
in {\em Spaces of Kleinian groups}, 
London Math. Soc. Lecture Note Ser. {\bf  329}(2006), 187--207.

\bibitem{hempel} J. Hempel, {\em $3$-manifolds}, 
Annals of Mathematics Studies, Number 86, Princeton University Press, 1976.

\bibitem{ito-selfbump} K. Ito, ``Convergence and divergence of Kleinian punctured torus groups,''
{\em Amer. J. Math.} {\bf 134}(2012), 861--889.

\bibitem{jaco-shalen} W. Jaco and P. Shalen, \emph{Seifert fibered spaces in $3$-manifolds}, 
Mem. Amer. Math. Soc. {\bf 21} (1979), no. 220.

\bibitem{johannson} K. Johannson, \emph{Homotopy Equivalences of $3$-manifolds with Boundary}, Lecture Notes in Mathematics, vol. 761, Springer-Verlag, 1979.


\bibitem{jorgenson-marden} T. J\o rgenson and A. Marden, ``Algebraic and geometric convergence of Kleinian groups,''
{\em Math. Scand.} {\bf 66}(1990), 47--72.

\bibitem{klarreich} E. Klarreich, ``The boundary at infinity of the curve complex
and the relative Teichm\"uller space,'' preprint.

\bibitem{kleineidam-souto} G. Kleineidam and J. Souto, ``Algebraic convergence of function groups,'' {\em Comm. Math. Helv.} 
{\bf 77}(2002), 244--269.

\bibitem{kra}  I. Kra, ``On spaces of Kleinian groups,'' {\em Comm. Math. Helv.} {\bf 47}(1972), 53--69.

\bibitem{kulkarni-shalen} R. Kulkarni and P. Shalen, ``On Ahlfors' finiteness
theorem,'' {\em Adv. Math.} {\bf 111}(1991), 155--169.

\bibitem{lecuire-masurdomain} C. Lecuire, ``An extension of the Masur domain,'' in {\em Spaces of Kleinian
Groups and Hyperbolic 3-manifolds}, L.M.S. Lecture Note Series, vol. {\bf 329}(2006), 49--74.

\bibitem{magid} A. Magid,  ``Deformation spaces of Kleinian surface groups
are not locally connected,''  {\em Geom. and Top.} {\bf 16}(2012), 1247--1320.

\bibitem{marden} A. Marden, ``The geometry of finitely generated Kleinian groups,'' {\em Annals of Math.} {\bf 99}(1974),
383--462.

\bibitem{maskit} B. Maskit, ``Self-maps of Kleinian groups,'' {\em Amer. J. Math.} {\bf 93}(1971), 840--856.

\bibitem{masur-minsky}
H. Masur and Y.~Minsky, ``Geometry of the complex of curves {I}:
Hyperbolicity,''  {\em Invent. Math.}  \textbf{138} (1999), 103--149.
  

\bibitem{mcculloughRCC} D. McCullough, ``Compact submanifolds of
3-manifolds with boundary,'' {\em Quart. J. Math. Oxford} {\bf 37}(1986), 299--306.

\bibitem{MMS} D. McCullough, A. Miller and G. A. Swarup, ``Uniqueness of cores of noncompact 
3-manifolds,'' {\em J. L.M.S.} {\bf 32}(1985), 548--556.

\bibitem{mcmullenCE} C. McMullen, ``Complex earthquakes and Teichm\"uller
theory,'' \emph{J. Amer. Math. Soc.} \textbf{11} (1998),  283--320.	
  
\bibitem{ELC1}
Y. Minsky, ``The classification of {Kleinian} surface groups {I}: models and
bounds,'' {\em  Annals of Math.} \textbf{171}(2010), 1--107.

\bibitem{morgan-shalenI} J. Morgan and P. Shalen, ``Degenerations of Hyperbolic structures I: Valuations, trees and surfaces,''
{\em Annals of Math.}  {\bf 120}(1984), 401--476.

\bibitem{namazi-souto} H. Namazi and J. Souto, ``Non-realizability, ending laminations and the density conjecture,''
{\em Acta Math.} {\bf 209}(2012), 323--395.

\bibitem{ohshika-gtlimits} K. Ohshika, ``Limits of geometrically tame Kleinian groups,''
{\em Invent. Math} {\bf 99}(1990), 185--203.

\bibitem{ohshika-density} K. Ohshika, ``Realising end invariants by limits of minimally
parabolic, geometrically finite groups,'' {\em Geom. Top.} {\bf 15}(2011), 827--890.

\bibitem{ohshika-selfbump} K. Ohshika, ``Divergence, exotic convergence and self-bumping
in quasiFuchsian spaces,'' preprint, available at: 
{\texttt http://front.math.ucdavis.edu/1010.0070}

\bibitem{otal-duke} J.P. Otal, ``Sur la d\'eg\'en\'erescence des groups de Schottky,'' {\em Duke Math. J.}
{\bf 74}(1994), 777--792.

\bibitem{otal} J. P. Otal, {\em Le Th\'eor\`eme d'Hyperbolisation pour les Vari\'et\'es Fibr\'ees de dimension 3}, 
{\em Asterique} {\bf 235}(1996).

\bibitem{skora} R. Skora, ``Splitting of surfaces,'' {\em Bull. A.M.S.} {\bf 23}(1990), 85--90.

\bibitem{sullivanII}  D. Sullivan, ``Quasiconformal homeomorphisms and dynamics. II. Structural stability implies
hyperbolicity for Kleinian groups,'' {\em Acta Math.} {\bf 155}(1985),  243--260.

\bibitem{thurston-notes} W. P. Thurston, {\em The Geometry and Topology of
Three-Manifolds}, Princeton University course notes, available
at http://www.msri.org/publications/books/gt3m/ (1980).

\bibitem{thurston-survey} W. Thurston, ``Three dimensional manifolds, Kleinian groups, and hyperbolic geometry,''
{\em Bull. A.M.S.} {\bf 6}(1982), 357--381.

\bibitem{thurstonII} W. Thurston, ``Hyperbolic structures on 3-manifolds, II: Surface groups and 3-manifolds
which fiber over the circle,'' preprint, available at arXiv:math.GT/9801045.

\bibitem{thurstonIII} W. Thurston, ``Hyperbolic structures on 3-manifolds, III: 
Deformations of 3-manifolds with incompressible boundary,'' preprint, available at:
\texttt{http://front.math.ucdavis.edu/math.GT/9801045}



}
\end{thebibliography}
\end{document}